\documentclass{amsart}
\usepackage{amsmath}
\usepackage{amsthm}
\usepackage{amsfonts}
\usepackage{amssymb}
\usepackage{bbold}
\usepackage{parskip}
\usepackage{mathrsfs}
\usepackage{amssymb}
\usepackage{graphicx}
\usepackage{float}
\usepackage{setspace}

\newcommand{\forces}{\vDash}
\newcommand{\G}{\Gamma}

\newcommand{\inv}{^{-1}}

\newcommand{\nlessdot}{\not\hspace{-4pt}\lessdot}
\newcommand{\ngtrdot}{\not\hspace{-2.75pt}\gtrdot}

\theoremstyle{definition}

\newtheorem{example}{Example}
\newtheorem{defn}{Definition}
\theoremstyle{theorem}
\newtheorem{lemma}{Lemma}
\newtheorem{prop}[lemma]{Proposition}
\newtheorem{theorem}[lemma]{Theorem}
\newtheorem{cor}[lemma]{Corollary}

\title{Universal Labeling Algebras as Invariants of Layered Graphs} 
\author{Susan Durst} 
\date{\today}

\begin{document} 

\begin{abstract}
In this work we will study the universal labeling algebra $A(\Gamma)$, a related algebra $B(\Gamma)$, and their behavior as invariants of layered graphs. We will introduce the notion of an upper vertex-like basis, which allows us to recover structural information about the graph $\Gamma$ from the algebra $B(\Gamma)$. We will use these bases to show that several classes of layered graphs are uniquely identified by their corresponding algebras $B(\Gamma)$.
\end{abstract}

\maketitle

\section{Introduction}
\label{chap:Introduction}

A directed graph $\Gamma=(V,E)$ consists of a set $V$ of vertices, together with a set $E$ of ordered pairs of elements of $V$ called edges.  Given an edge $e=(v,w)$, we call $v$ the ``tail'' of $e$ and write $v=t(e)$, and we call $w$ the ``head'' of $e$ and write $w=h(e)$.  A \textit{path} is a sequence of edges $(e_1,e_2,\ldots,e_n)$ satisfying $h(e_i)=t(e_{i+1})$ for all $1\leq i<n$.  A \textit{layered graph} is a directed graph $\Gamma=(V,E)$ whose vertex set $V$ is divided into a sequence of layers $V_0, V_1, V_2,V_3,\ldots$ such that each directed edge in $E$ travels exactly one layer down.

Let $\Gamma=(V,E)$ be a layered graph, let $A$ be an algebra over some field $\mathbb{F}$, and let $f$ be a map from $E$ to $A$.  To each path $\pi=(e_1,\ldots,e_n)$ in $\Gamma$, we associate a polynomial $p_\pi=(t-f(e_1))(t-f(e_2))\ldots(t-f(e_n))$ in $A[t]$.  We say that the ordered pair $(A,f)$ is a $\Gamma$-\textit{labeling} of $\Gamma$ if $p_{\pi_1}=p_{\pi_2}$ for any paths $\pi_1$ and $\pi_2$ with the same starting and ending vertices.  Each graph $\Gamma$ has a universal $\Gamma$-labeling, given by $(A(\Gamma),f_\Gamma)$.  The algebra $A(\Gamma)$ is called the \textit{universal labeling algebra} for the graph $\Gamma$.

These algebras arose from the constructions occurring in the proof of Gelfand and Retakh's Vieta theorem~\cite{Vieta} for polynomials over a noncommutative division ring. There they show how to write such a polynomial in a central variable $t$ with a specified set of roots $\{x_i \mid 1\leq i\leq n\}$ in the form
\[f(t)=(t-y_1)(t-y_2)\ldots(t-y_n).\]
The expressions for the $y_i$ depend on the ordering of the roots, and lead to labelings of the Boolean lattice.

Universal labeling algebras induce an equivalence relation $\sim_A$ on the collection of layered graphs, given by $\Gamma\sim_A\Gamma^\prime$ if and only if $A(\Gamma)\cong A(\Gamma^\prime)$.  We would like to explore the equivalence classes of this relation.  In order to do so, we will work with a related algebra $B(\Gamma)$, generated by the vertices with nonzero out-degree, and with relations generated by
\[\left\{vw: (v,w)\notin E\right\}\cup\left\{v\sum_{vw\in E}w\right\}\]
In the case where $\Gamma$ is a \textit{uniform} layered graph as defined in Section~\ref{section:Uniform}, $B(\Gamma)$ can be calculated directly from $A(\Gamma)$, and thus gives us a coarser partition of the layered graphs. 
In Sections 2-\ref{section:Uniform}, we will give a full development of the relationship between the algebras $A(\Gamma)$ and $B(\Gamma)$.

The $\sim_B$ equivalence class of a particular layered graph is often nontrivial.  For example, consider the pair of graphs in the figure below:

\begin{figure}[H]
\caption{The graphs $\Gamma$ and $\Gamma^\prime$ are nonisomorphic, but we have $B(\Gamma)\cong B(\Gamma^\prime)$.}
\begin{center}
\includegraphics[scale=.15]{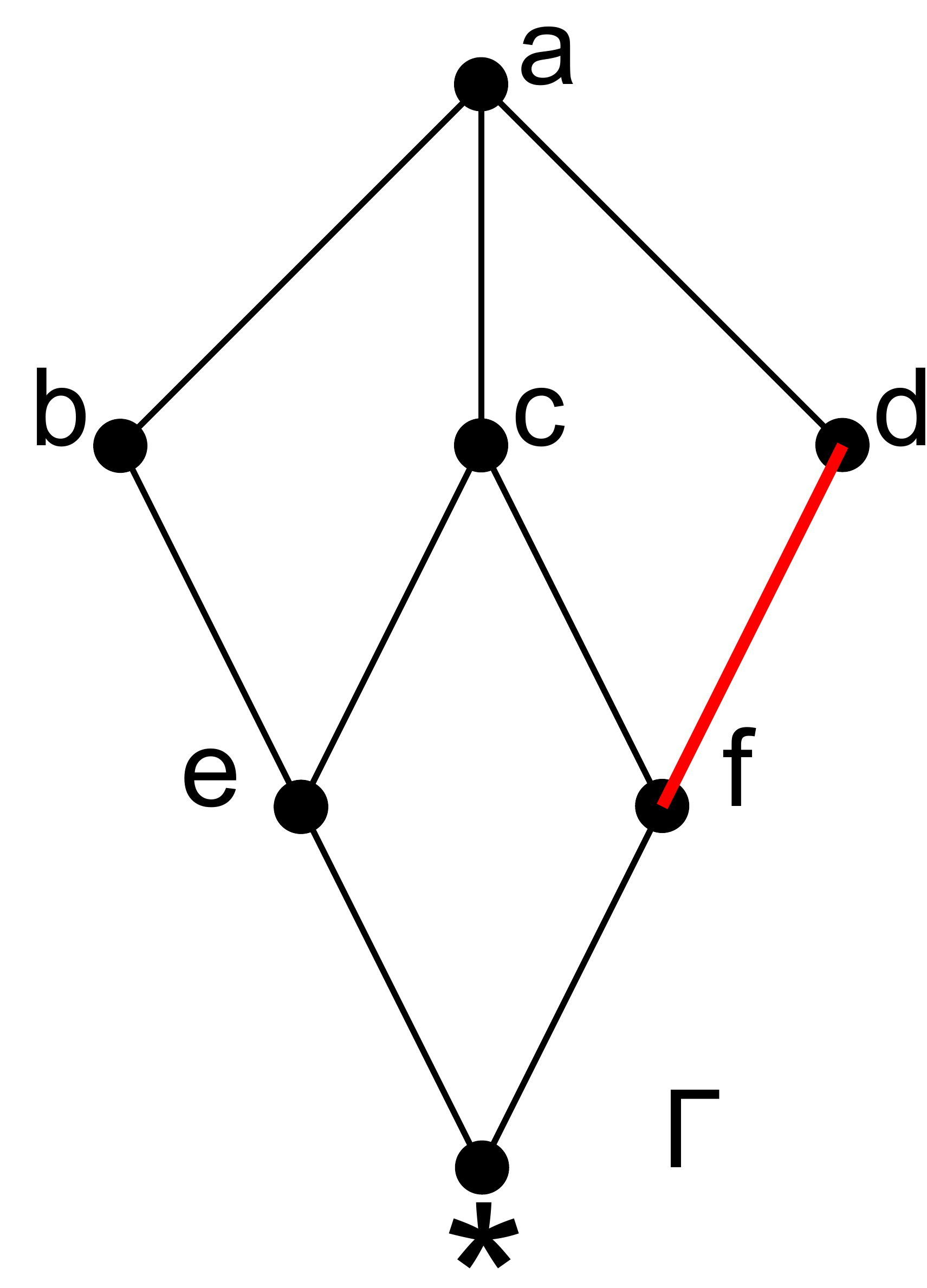}
\hspace{24pt}
\includegraphics[scale=.15]{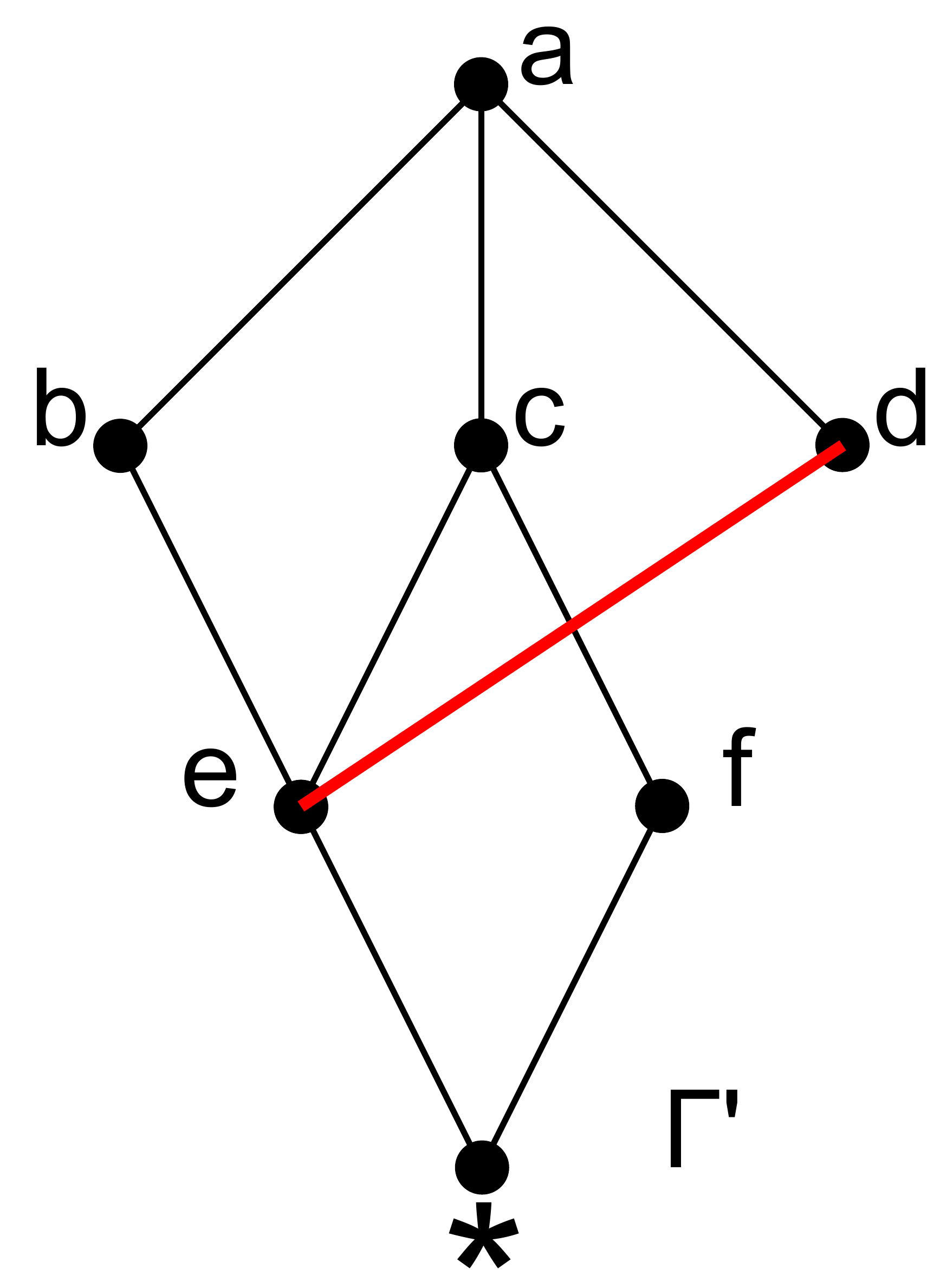}
\end{center}
\end{figure}

These graphs are nonisomorphic, but the only difference is the single edge $(d,f)$ in $\Gamma$, replaced by the edge $(d,e)$ in $\Gamma^\prime$.  This replacement has no effect on the defining relations for $B(\Gamma)$, and thus we have $\Gamma\sim_B\Gamma^\prime$. In Sections \ref{chap:IsosOfB}-\ref{chap:UniquenessResults} we will explore the structure of $B(\Gamma)$, discuss when two graphs $\Gamma$ and $\Gamma^\prime$ satisfy $\Gamma\sim_B\Gamma^\prime$, and give examples of several classes of uniform layered graphs which have $\sim_B$, and thus $\sim_A$ equivalence classes consisting of a single isomorphism class of layered graphs.  These include complete layered graphs, Boolean lattices, and lattices of subspaces of finite-dimensional vector spaces over finite fields.

\section{Preliminaries}

\subsection{Ranked Posets and Layered Graphs}

Recall that a \textit{directed graph} $G=(V,E)$ consists of a set $V$ of vertices, together with a set $E$ of ordered pairs of vertices called \textit{edges}.  We define functions $t$ and $h$ from $E$ to $V$ called the \textit{tail} and \textit{head} functions, respectively, such that for $e=(v,w)$, we have $t(e)=v$ and $h(e)=w$.

In \cite{AlgebrasAssocToDirGraphs}, Gelfand, Retakh, Serconek, and Wilson define a \textit{layered graph} to be a directed graph $\Gamma=(V,E)$ such that $V=\bigsqcup_{i=0}^n V_i$, and such that whenever $e\in E$ and $t(e)\in V_i$, we have $h(e)\in V_{i-1}$.  For any $v\in V$, define $S(v)$ to be the set $\{h(e): t(e)=v\}$.  We will use the notation $V_{\geq k}$ for $\bigsqcup_{i=k}^n V_i$, and $V_+$ for $V_{\geq 1}$.  The layered graphs that we consider here all have the property that for any $v\in V_+$, the set $S(v)$ is nonempty.

\begin{figure}[h]
\caption{A layered graph with layers $V_0$, $V_1$, $V_2$, and $V_3$.}

\begin{center}
\includegraphics[scale=.35]{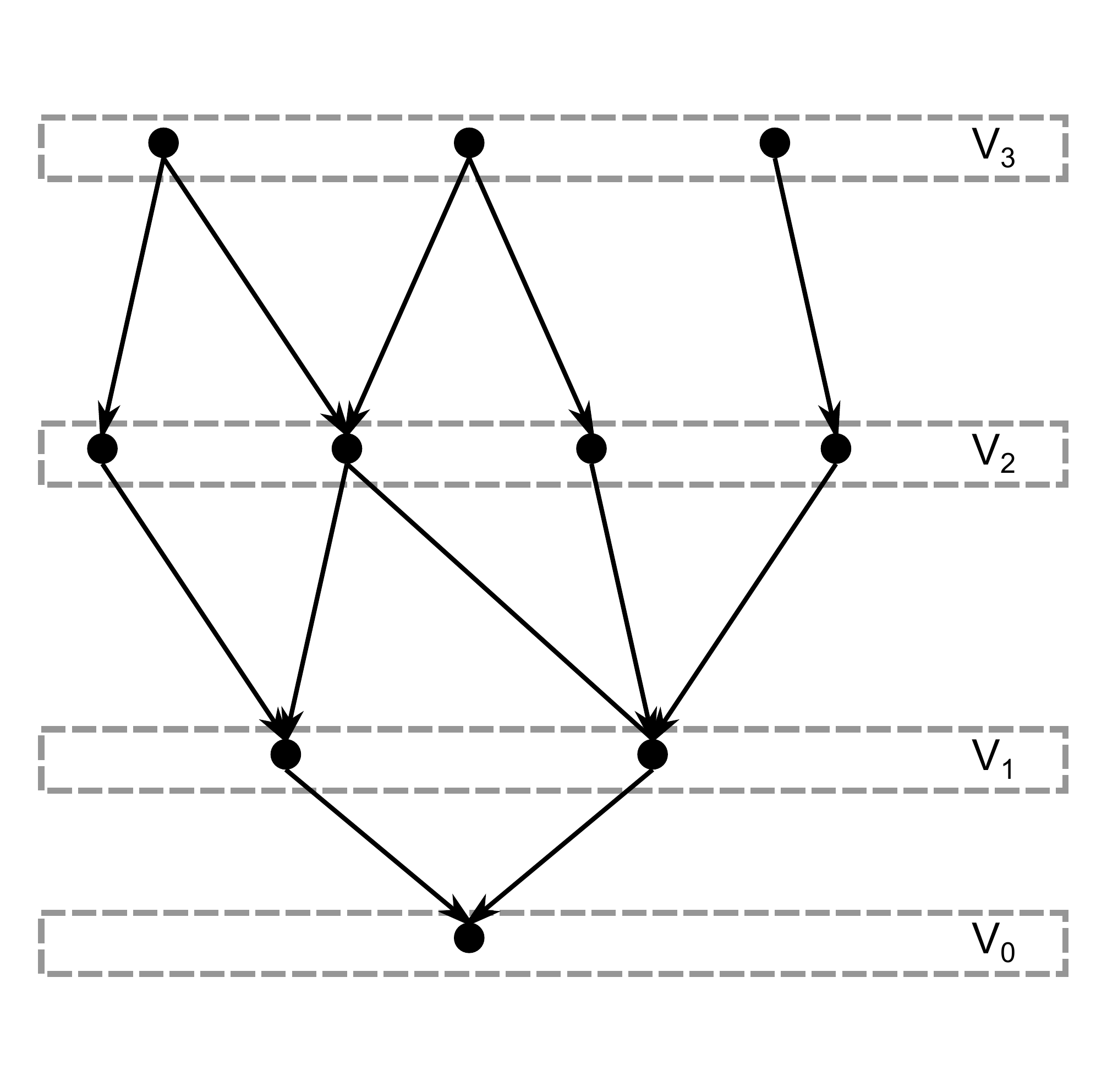}
\end{center}
\end{figure}

We recall some notation from the theory of partially ordered sets.  If $(P,\leq)$ is a poset and $x,y\in P$, we write $x\gtrdot y$ and say that $x$ \textit{covers} $y$ if the following holds:
\begin{itemize}
\item[(i)] $x> y$.
\item[(ii)] For any $z\in P$, such that $x\geq z\geq y,$ we have $x=z$ or $y=z$.
\end{itemize}
A \textit{ranked poset} is a partially ordered set $(P,\leq)$ together with rank function $|\cdot|:P\rightarrow\mathbb{N}$ satisfying the following two properties:
\begin{itemize}
\item[(i)]$(x\gtrdot y)\Rightarrow(|x|=|y|+1)$.
\item[(ii)]$|x|=0$ if and only if $x$ is minimal in $P$.
\end{itemize}

Notice that if $\Gamma=(V,E)$ is a layered graph as defined above, then we can associate to $\Gamma$ a partially ordered set $(V,\leq)$, with the partial order given by $w\leq v$ if and only if there exists a directed path in $\Gamma$ from $v$ to $w$.  If $\Gamma$ satisfies the additional property that $S(v)=\emptyset$ if and only if $v\in V_0$, then this is a ranked poset, with rank given by $|v|=i$ if and only if $v\in V_i$.  When it is convenient, we will treat the graph $\Gamma$ as a ranked partially ordered set of vertices, writing $v\geq w$ to indicate that there is a directed path from $v$ to $w$, $v\gtrdot w$ to indicate that there is a directed edge from $v$ to $w$, and $|v|=i$ to indicate that $v\in V_i$.

Also notice that any ranked poset $(P,\gtrdot)$ has as its Hasse diagram a layered graph $(P,E_P)$, with edges given by $(p,q)$ for each covering relation $p\gtrdot q$.  When it is convenient, we equate $P$ with its Hasse diagram and refer to the algebra $A(P)$ for any ranked poset $P$.  There is a one-to-one correspondence between layered graphs of this type and ranked posets.

We will find the following notation useful in our discussion of layered graphs:

\begin{defn}
For any subset $T\subseteq V_n$, we define the following:
\begin{itemize}
\item[(i)]$S(T)=\bigcup_{t\in T}S(t)$, the set of all vertices covered by some vertex in $T$.
\item[(ii)] $\sim_T$ is the equivalence relation obtained by taking the transitive closure of the relation 
\[R_T=\{(v,w) : \text{There exists } t\in T\text{ with } v,w\in S(t)\}.\]
\item[(iii)] $\mathscr{C}_T$ is the collection of equivalence classes of $V_{n-1}$ under $\sim_T$.
\item[(iv)] $k_T=|\mathscr{C}_T|$.
\item[(v)] $k_T^T=k_T-|V_{n-1}|+|S(T)|$.
\end{itemize}
\end{defn}

Notice that $k_T^T$ is the number of equivalence classes of $S(T)$ under $\sim_T$.  Since $\emptyset\subseteq V_n$ for multiple values of $n$, we will use the notation $\emptyset_n$ to indicate that $\emptyset$ is being considered as a subset of $V_n$.

\begin{prop}
If $\G$ is a layered graph, and $A\subseteq V_n$, then we have the following:
\begin{itemize}
\item[(i)]$|V_{n-1}|=k_{\emptyset_n}$
\item[(ii)]$\left|V_{n-1}\setminus S(A)\right|=k_A-k_A^A,$
\item[(iii)]$|S(A)|=k_{\emptyset_n}-k_A+k_A^A$
\end{itemize}
\end{prop}

\begin{proof}
Statement (i) follows immediately from the definition of $k_A$.  Statement (ii) follows from the definition of $k_A^A$.  Statement (iii) follows easily from statements (i) and (ii).
\end{proof}

\subsection{Universal Labeling Algebras}

Let $\Gamma=(V,E)$ be a layered graph.   We call an ordered $n$-tuple of edges $\pi=(e_1,e_2,\ldots,e_n)$ a \textit{path} if $h(e_i)=t(e_{i+1})$ for all $1\leq i<n$. We define $t(\pi)=t(e_1)$, and $h(\pi)=h(e_n)$.  We write $||\pi||=n$, and say that the \textit{length} of the path $\pi$ is equal to $n$.

We will occasionally find it useful to refer to the notion of a \textit{vertex path}, a sequence $(v_1,\ldots,v_n)$ of vertices such that for every $1\leq i<n$ we have $(v_i,v_{i+1})\in E$. Each path $(e_1,\ldots,e_n)$ has an associated vertex path given by $(t(e_1),t(e_2),\ldots,t(e_n),h(e_n))$, and each vertex path $(v_1,\ldots, v_n)$ has an associated path $((v_1,v_2),(v_2,v_3),\ldots,(v_{n-1},v_n))$.

For two paths $\pi$ and $\pi^\prime$ with $h(\pi^\prime)=t(\pi)$, we define
\[\pi\wedge\pi^\prime=(e_1,\ldots,e_n,f_1,\ldots,f_m).\]
Sometimes for a single edge $f$, we will write $f\wedge\pi$ for $(f,e_1,\ldots,e_n)$ or $\pi\wedge f$ for $(e_1,\ldots,e_n,f)$.

For two paths $\pi$ and $\pi^\prime$, we will write $\pi\approx\pi^\prime$ if $\pi$ and $\pi^\prime$ start and end at the same vertices.  That is, $\pi\approx\pi^\prime$ if $||\pi||=||\pi^\prime||$, $t(\pi)=t(\pi^\prime)$, and $h(\pi)=h(\pi^\prime)$.

\begin{figure}[h]
\begin{center}
\caption{The two paths $\pi_1=(a_1,\ldots,a_n)$ and $\pi_2=(b_1,\ldots,b_n)$ start and end at the same vertices, and so $\pi_1\approx\pi_2$.}
\vspace{6pt}
\includegraphics[scale=.3]{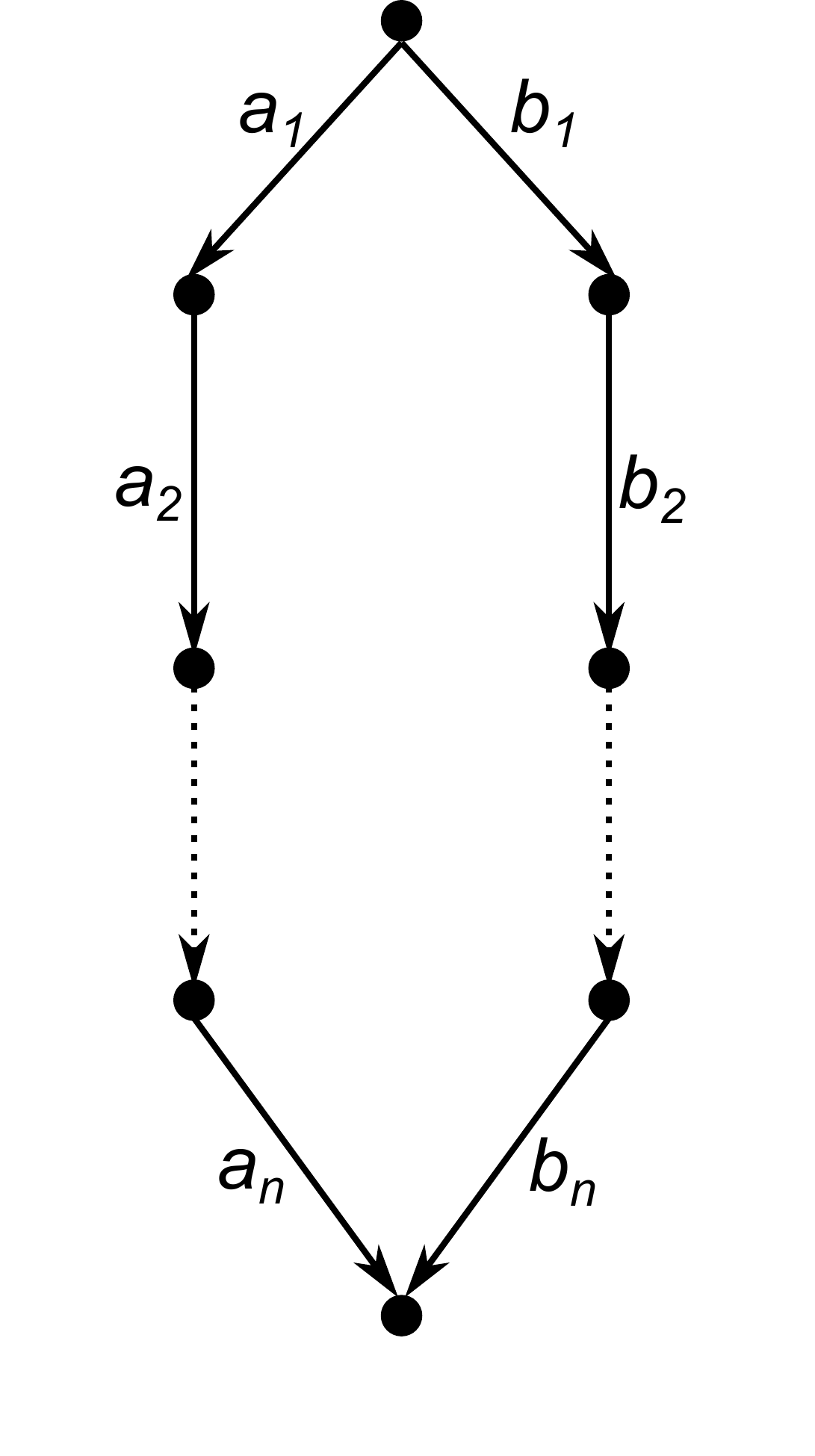}
\end{center}
\end{figure}

Let $\Gamma=(V,E)$ be a layered graph, and fix a field $\mathbb{F}$.  Let $A$ be an $\mathbb{F}$-algebra, and let $f:E\rightarrow A$ be a set map.  To each path $\pi=(e_1,\ldots,e_n)$ in $\Gamma$, we associate a polynomial 
\[p_{f,\pi}=(t-f(e_1))(t-f(e_2))\ldots(t-f(e_n))\in A[t],\]
where $t$ is a central indeterminate.  

\begin{defn}
The ordered pair $(A,f)$ is called a \textit{$\Gamma$-labeling} if it satisfies $p_{f,\pi}=p_{f,\pi^\prime}$ whenever $\pi\approx\pi^\prime$.  Given two $\Gamma$-labelings $(A,f)$ and $(A^\prime,f^\prime)$, the algebra map $\phi:A\rightarrow A^\prime$ is a \textit{$\Gamma$-labeling map} if $\phi\circ f=f^\prime.$
\end{defn}


For each layered graph $\Gamma$, we define an algebra $A(\Gamma)$  as follows: let $T(E)$ be the free associative algebra on $E$ over $F$, and to each path $\pi=(e_1,\ldots,e_n)$, associate the polynomial 
\[P_\pi=(t-e_1)(t-e_2)\ldots(t-e_n)=\sum_{i=0}^\infty e(\pi,i)t^i\in T(E)[t].\]
Define $A(\Gamma)$ to be $T(E)/R$, where $R$ is the ideal generated by the set
\[\{e(\pi,i)-e(\pi^\prime,i) \mid \pi\approx\pi^\prime\text{ and }0\leq i\leq ||\pi||\}.\]
This is precisely the algebra introduced by Gelfand, Retakh, Serconek, and Wilson in \cite{AlgebrasAssocToDirGraphs}.

\begin{prop}
Let $\Gamma=(V,E)$ be a layered graph, let $A(\Gamma)=T(E)/R$ as defined above, and let $f_\Gamma:E\rightarrow A(\Gamma)$ be defined by $f_\Gamma(e)=e+R$.  Then $(A(\Gamma),f_\Gamma)$ is the universal  $\Gamma$-labeling.
\end{prop}

\begin{proof}
Let $(A,f)$ be an arbitrary $\Gamma$-labeling.  We need to show that there exists a unique $\Gamma$-labeling map from $A(\Gamma)$ to $A$.  Let $g:E\rightarrow T(E)$ be the canonical embedding of $E$ into the free algebra $T(E)$.  Then there exist unique algebra isomorphisms $\psi_{A(\Gamma)}:T(E)\rightarrow A(\Gamma)$ and $\psi_{A}:T(E)\rightarrow A(\Gamma)$ satisfying $\psi_{A(\Gamma)}\circ g=f_\Gamma$ and $\psi_{A}\circ g=f$.  We know that $\psi_{A(\Gamma)}$ is surjective, with kernel $R$.

Notice that for any path $\pi$ in $\Gamma$, we have
\[p_{f,\pi}=\sum_{i=0}^\infty f(e(\pi,i))t^i\in A[t].\]
Since $p_{f,\pi_1}=p_{f,\pi_2}$ in $A[t]$ for any $\pi_1\approx\pi_2$, it follows that $f(e(\pi_1,i))=f(e(\pi_2,i))$ in $A$ whenever $\pi_1\approx \pi_2$ and $0\leq i\leq ||\pi_1||$.  Thus $e(\pi_1,i)-e(\pi_2,i)\in\ker(\psi_A)$ for all $\pi_1\approx \pi_2$ and $0\leq i\leq ||\pi_1||$.  This gives us $R\subseteq\ker\psi_A$, and so $\psi_A$ factors uniquely through $A(\Gamma)$.  Thus there exists a unique $\Gamma$-labeling map from $A(\Gamma)$ to $A$, and our proof is complete.
\end{proof}

\begin{defn}
With $R$ defined as above, we define $R_i\subseteq R$ to be the ideal generated by 
\[\{e(\pi,k)-e(\pi^\prime,k)\mid \pi\approx\pi^\prime\text{ and }0\leq k\leq \min\{i,||\pi||\}\}\]
We define $A(\Gamma,i)=T(E)/R_i.$
\end{defn}

For a layered graph $\Gamma$ with $n+1$ nonempty levels, we have
\[R_1\subseteq R_2\subseteq\ldots\subseteq R_n=R,\]
and thus there exist canonical maps
\[T(E)\rightarrow A(\Gamma,1)\rightarrow \ldots\rightarrow A(\Gamma,n)=A(\Gamma).\]

\section{Presentation of $A(\Gamma)$ as a Quotient of $T(V_+)$}
\label{vertpres}
In the case where $\Gamma=(V,E)$ has a unique minimal vertex $*$, we can find a presentation of $A(\Gamma)$ as a quotient of $T(V_+)$, the free algebra generated by vertices in $V_+=(V\setminus V_0)$.  Our proof differa from the proof presented in \cite{KozulPaper}, in that we focus on the following proposition:

\begin{prop}
\label{AGamma1}
Let $\Gamma$ be a layered graph with unique minimal vertex $*$.  Then $A(\Gamma,1)\cong T(V_+)$.
\end{prop}

To understand why this is the case, it is useful to consider the  generating relations for $R_1$, and the structure of $\Gamma$ when it has a unique minimal vertex.  We know that $R_1$ is generated by the set
\[\{e(\pi,n-1)-e(\pi^\prime,n-1) : \pi\approx\pi^\prime, ||\pi||=n\}\]
If $\pi=(a_1,\ldots,a_n)$ and $\pi^\prime=(b_1,\ldots,b_n)$, then $e(\pi,n-1)$ and $e(\pi^\prime,n-1)$ are the $(t^{n-1})$-coefficients of $P_{\pi}=(t-a_1)\ldots(t-a_n)$ and  $P_{\pi^\prime}=(t-b_1)\ldots(t-b_n)$ respectively.  Thus $e(\pi,n-1)=a_1+\ldots+a_n$ and $e(\pi^\prime,n-1)=b_1+\ldots+b_n$.  It follows that $R_1$ is generated by the collection of polynomials of the form
\[(a_1+\ldots+a_n)-(b_1+\ldots+b_n),\]
where $\pi=(a_1,\ldots,a_n)$ and $\pi^\prime=(b_1,\ldots,b_n)$ are paths satisfying $\pi\approx \pi^\prime$.

Since $A(\Gamma,1)$ is a quotient of $T(E)$, and we wish to find an isomorphism between $A(\Gamma,1)$ and $T(V_+)$, it makes sense to find some way of associating the vertices in $V_+$ with edges or paths in $\Gamma$.  In the case where $\Gamma$ has a unique minimal vertex, this is easy.

\begin{defn}
Let $\Gamma=(V,E)$ be a layered graph with unique minimal vertex $*$.  For each vertex $v\in V_+$, we choose a distinguished edge $e_v=(v,w)\in E$.  We define a path
\[\pi_v=\left(e_1,e_2,\ldots,e_{|v|}\right),\]
where $e_1=e_v$, and $e_i=e_{h\left(e_{i-1}\right)}$ for $1<i\leq |v|$.  When there is no possibility of ambiguity, we will also use the symbol $\pi_v$ to designate the corresponding vertex path $\left(t(e_1),t(e_2),\ldots,t\left(e_{|v|}\right)\right)$.
\end{defn}

In the proof of Proposition~\ref{AGamma1}, we will associate each vertex $v\in V_+$ to the sum of the edges in $\pi_v$.  This collection of elements will be sufficient to generate $A(\Gamma,1)$.  

\begin{figure}[h]
\caption{We have $\pi_v\approx f\wedge\pi_w$, since both paths start at $v$ and end at $*$.}
\begin{center}
\vspace{6pt}
\includegraphics[scale=.2]{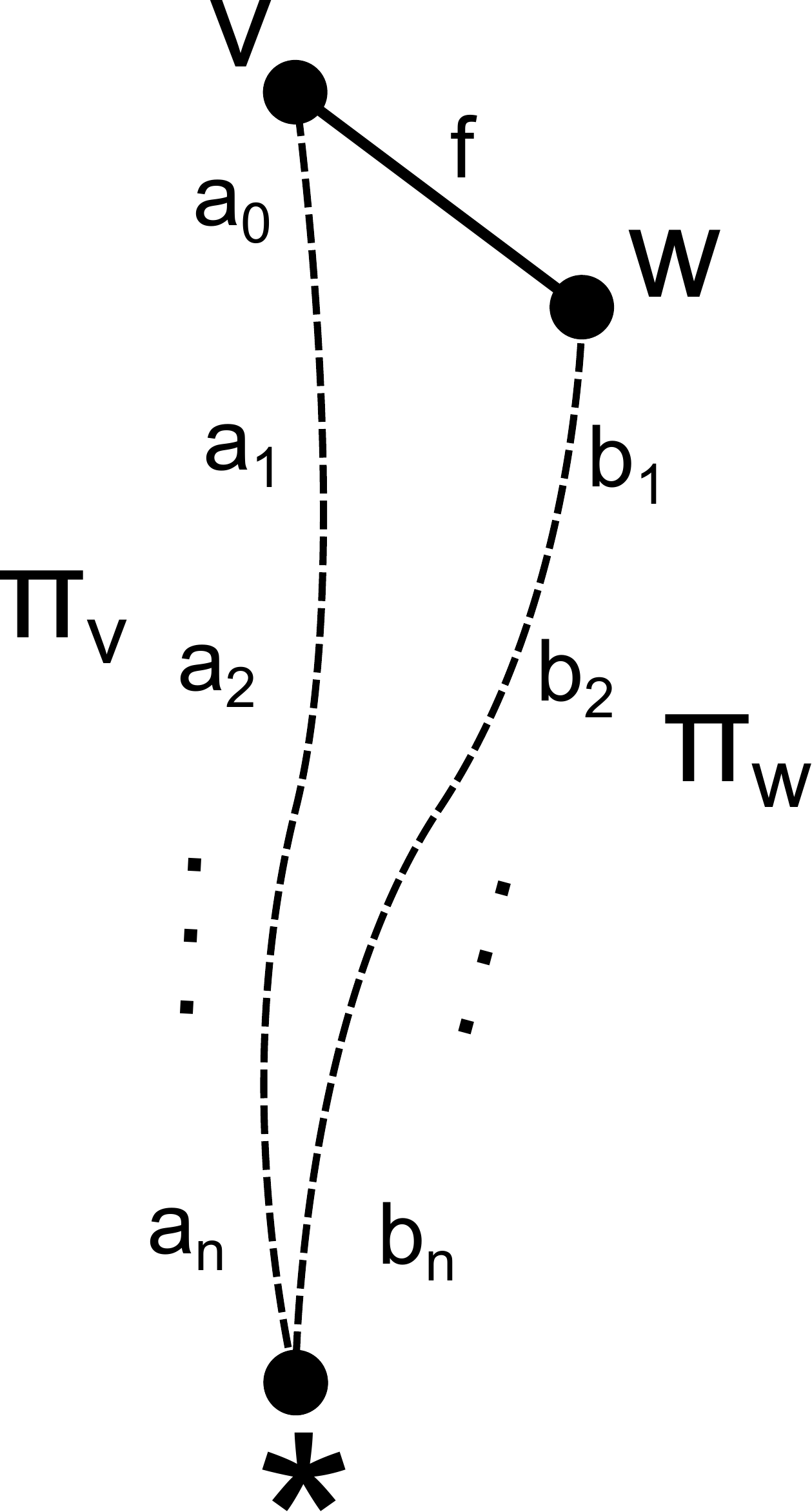}
\end{center}
\end{figure}

To see this, consider an arbitrary edge $f=(v,w)$, with $\pi_v=(a_0,\ldots,a_n)$ and $\pi_w=(b_1,\ldots,b_{n-1})$.  We have $\pi_v\approx f\wedge\pi_w$.  This means that
\[(a_0+\ldots+a_n)-(f+b_1+\ldots+b_{n-1})\in R_1\]
It follows that
\[f-(a_0+\ldots+a_n)-(b_1+\ldots,b_{n-1})\in R_1,\]
and so in $A(\Gamma,1)$, the edge $f$ corresponds to the element we have chosen to associate with $v-w$.  This gives us our isomorphism.  With this in mind, we proceed to the formal proof.

\begin{proof}[Proof of Proposition~\ref{AGamma1}]
Let $\phi^\prime$ be the canonical quotient map from $T(E)$ to $A(\Gamma,1)$ with quotient $R_1$.  Let $\phi^{\prime\prime}$ be the map from $T(E)$ to $T(V_+)$ be defined by
\[\phi^{\prime\prime}((v,w))=\left\{
\begin{array}{ll}
v-w&\text{if }w\neq *\\
v&\text{if }w=*
\end{array}\right.\]

\textbf{Claim 1:} $\phi^{\prime\prime}$ factors through $\phi^\prime$.

To prove this, we must show that $R_1\subseteq \ker(\phi^{\prime\prime})$.  Given our discussion above, it is sufficient to show that for any $\pi_1=(a_1,\ldots,a_n)$ and $\pi_2=(b_1,\ldots,b_n)$ with $\pi_1\approx\pi_2$, we have
\[\phi^{\prime\prime}((a_1+\ldots+a_n)-(b_1+\ldots+b_n))=0.\]

The expression $\phi^{\prime\prime}((a_1+\ldots+a_n)-(b_1+\ldots+b_n))$ is equal to
\[\left((t(a_1)-h(a_1))+\ldots+(t(a_n)-h(a_n))\right)-\left((t(b_1)-h(b_1))+\ldots+(t(b_n)-h(b_n))\right),\]
where $h(f)=0$ whenever $f=(v,*)$.  Since $h(a_i)=t(a_{i+1})$ and $h(b_i)=t(b_{i+1})$ for $1\leq i<n$, this is equal to
\[t(a_1)-h(a_n)-t(b_1)+h(b_n)\]
Since $\pi_1\approx\pi_2$, we have $t(a_1)=t(b_1)$ and $h(a_n)=h(b_n)$, and so this expression is equal to zero, proving Claim 1.

It follows that there exists a map $\phi: A(\Gamma,1)\rightarrow T(V_+)$ satisfying $\phi\circ\phi^\prime=\phi^{\prime\prime}$.  We define a map $\psi:T(V_+)\rightarrow A(\Gamma,1)$ such that for each $v\in V_+$ with $\pi_v=(a_1,\ldots,a_n)$, we have
\[\psi(v)=\phi^\prime(a_1+\ldots+a_n).\]
We wish to prove that $\psi=\phi^{-1}$.

\textbf{Calim 2:} $\psi\circ\phi=\text{id}_{A(\Gamma,1)}$.

Since $A(\Gamma,1)$ is generated by $\phi^{\prime}(E)$, it will suffice to show that $\psi\circ\phi$ fixes the images of all edges.  Let $f=(v,w)$, $w\neq *$, $\pi_v=(a_0,\ldots,a_n)$, and $\pi_w=(b_1,\ldots,b_{n-1})$.  Then
\begin{eqnarray*}
(\psi\circ\phi)(\phi^\prime(f))&=&\psi(\phi^{\prime\prime}(f))\\
&=&\psi(v-w)\\
&=&\phi^\prime((a_0+\ldots+a_n)-(b_1+\ldots+b_{n-1}))
\end{eqnarray*}
From the discussion above, we know that
\[((a_0+\ldots+a_n)-(b_1+\ldots+b_{n-1}))-f\in R_1.\]
It follows that $(\psi\circ\phi)(\phi^\prime(f))=\phi^\prime(f)$ for all $f=(v,w)$ with $w\neq *$.

If $f=(v,*)$, then
\[(\psi\circ\phi)(\phi^\prime(f))=\psi(\phi^{\prime\prime}(f))=\psi(v)=\phi^\prime(f).\]
Thus $\psi\circ\phi$ fixes $\phi^\prime(f)$ for all $f\in E$ and Claim 2 is proved.

\textbf{Claim 3:} $\phi\circ\psi=\text{id}_{T(V_+)}$.

Since $T(V_+)$ is generated by $V_+$ it will suffice to show that $\phi\circ\psi$ fixes $V_+$.  Let $v\in V_+$, and let $\pi_v=(a_1,\ldots,a_n)$.  Then
\begin{eqnarray*}
(\phi\circ\psi)(v)&=&\phi(\phi^\prime(a_0+\ldots+a_n))\\
&=&\phi^{\prime\prime}(a_0+\ldots+a_n)\\
&=&(t(a_0)-h(a_0))+\ldots+(t(a_{n-1})-h(a_{n-1}))+t(a_n)\\
&=&t(a_0)\\
&=&v
\end{eqnarray*}
This proves Claim 3.  Thus $\psi=\phi^{-1}$, and our result is proved.
\end{proof}

\begin{defn}
Let $\phi^{\prime\prime}$ be the map described in the proof of Proposition~\ref{AGamma1}.  For ease of notation, for every $f\in E$, we will set $\tilde{f}=\phi^{\prime\prime}(f)$, and $\tilde{e}(\pi,k)=\phi^{\prime\prime}(e(\pi,k))$.
\end{defn}

\begin{cor}
$A(\Gamma)=T(V_+)/R_V$, where $R_V$ is the ideal generated by
\[\left\{\tilde{e}(\pi,i)-\tilde{e}(\pi^\prime,i) \mid \pi\approx\pi^\prime\text{ and }0\leq i\leq ||\pi||\right\}.\]
\end{cor}

\begin{proof}
Since $\phi: A(\Gamma,1)\rightarrow T(V_+)$ is an isomorphism, the induced map from $A(\Gamma,1)/\phi^\prime(R)$ to $T(V_+)/(\phi\circ\phi^\prime)(R)=T(V_+)/\phi^{\prime\prime}(R)$ is also an isomorphism.  It is clear that $\phi^{\prime\prime}(R)=R_V$, and so
\[T(V_+)/R_V\cong A(\Gamma,1)/\phi^\prime(R)\cong A(\Gamma)\]
\end{proof}

\subsection{$A(\Gamma)$ as a Layered Graph Invariant}

To each layered graph $\Gamma$, we have associated an algebra $A(\Gamma)$.  It is natural to ask how this algebra behaves when considered as an invariant of layered graphs.  However, if we only consider its structure as an algebra, it is not a particularly strong invariant.  Consider the two graphs below:

\begin{center}
\includegraphics[scale=.3]{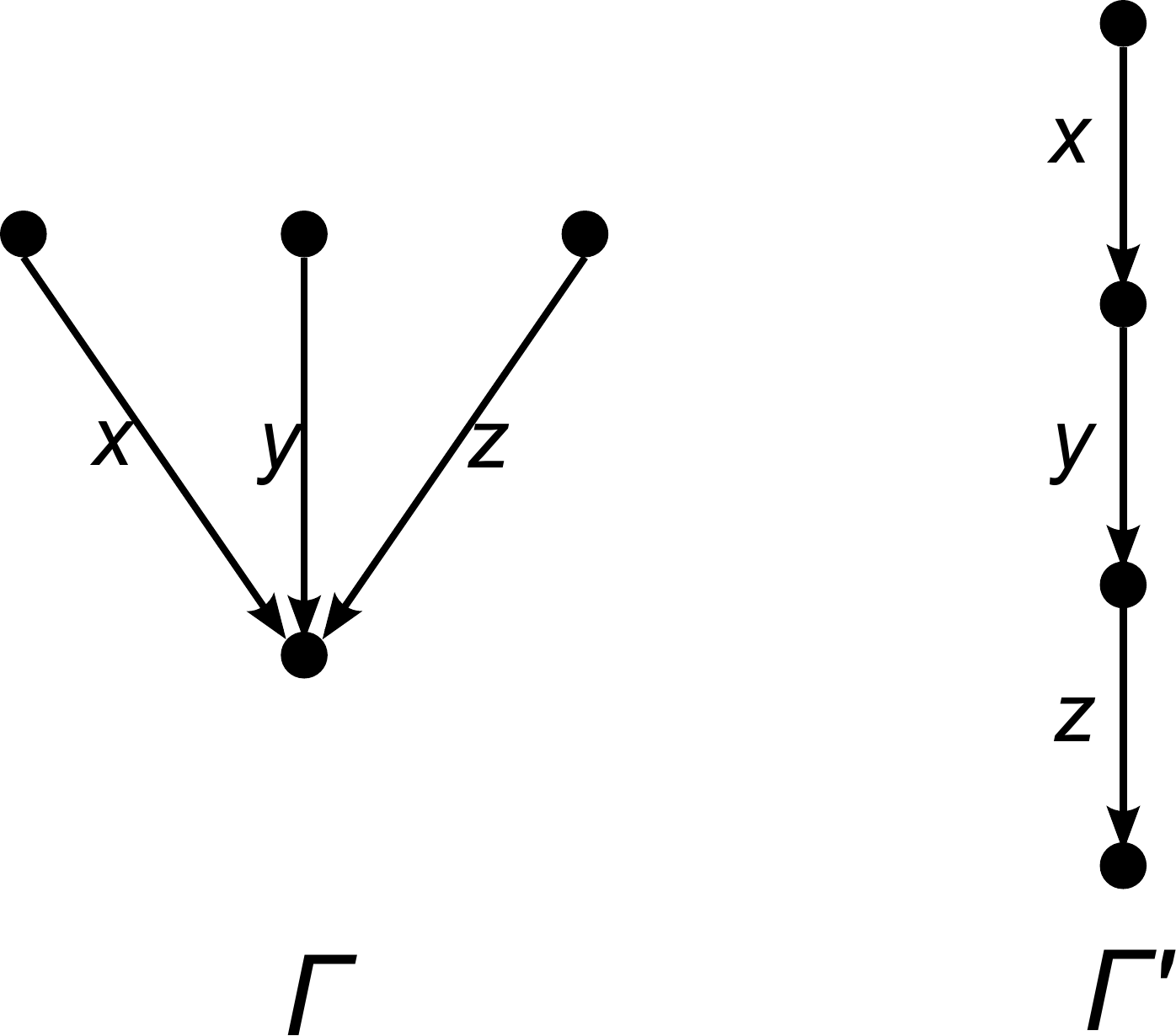}
\end{center}

We have $A(\Gamma)\cong A(\Gamma^\prime)\cong F\langle x,y,z\rangle$.  This is unfortunate, as one would hope that an invariant of layered graphs would be able to distinguish between graphs with different numbers of layers, or graphs whose layers have a different number of vertices.  The algebra $A(\Gamma)$ clearly does neither.  To capture these properties, we will need to consider some additional structure on $A(\Gamma)$.

Notice that the ideal $R_V$ is homogeneous with respect to the degree grading of $T(V_+)$.  Thus there is a grading on $A(\Gamma)$ given by $A(\Gamma)=\bigoplus A(\Gamma)_{[i]},$ with
\[A(\Gamma)_{[i]}=span\{v_1v_2\ldots v_i : v_1,\ldots,v_i\in V_+\}.\]
We will call this the \textit{degree grading} of $A(\Gamma)$.

$A(\Gamma)$ also has a filtration given by the level of the vertices in the graph.  So $A(\Gamma)=\bigcup A(\Gamma)_{i}$, where
\[A(\Gamma)_{i}\leq span\left\{v_1v_2\ldots v_j : \sum_{k=1}^j|v_k|\leq i\right\}.\]
We call this the \textit{vertex filtration} on $A(\Gamma)$.  When we consider the algebra $A(\Gamma)$ together with the degree grading and the vertex filtration, we obtain an invariant that can distinguish between graphs that have layers of different sizes, as we will see in the sections that follow.

\begin{defn}
We say $\Gamma\sim_A\Gamma^\prime$ if and only if there exists an isomorphism $\phi:A(\Gamma)\rightarrow A(\Gamma^\prime)$ which preserves both the degree grading and the vertex filtration.
\end{defn}

When we allow ourselves to consider these pieces of structure, we obtain an invariant that can distinguish between graphs that have layers of different sizes.

\begin{prop}
If $\Gamma=(V,E)$, $\Gamma^\prime=(V^\prime,E^\prime)$, and $\Gamma\sim_A\Gamma^\prime$, then $|V_i|=|V_i^\prime|$ for all $i\in\mathbb{N}$.
\end{prop}

\begin{proof}
We have
\[|V_i|=\dim\left(\left(A(\Gamma)_{[1]}\cap A(\Gamma)_{i}\right)/\left(A(\Gamma)_{[1]}\cap A(\Gamma)_{i-1}\right)\right)\]
and
\[|V_i^\prime|=\dim\left(\left(A(\Gamma^\prime)_{[1]}\cap A(\Gamma^\prime)_i\right)/\left(A(\Gamma^\prime)_{[1]}\cap A(\Gamma^\prime)_{i-1}\right)\right).\]
Any isomorphism $\phi:A(\Gamma)\rightarrow A(\Gamma^\prime)$ that preserves the grading and the filtration will map the subspace from the first expression onto the subspace in the second expression.  The result follows.
\end{proof}

\section{The Associated Graded Algebra $gr(A(\Gamma))$}

\subsection{Associated Graded Algebras}
\label{gramma}
Let $V$ be a vector space with filtration \[V_0\subseteq V_1\subseteq V_2\subseteq\ldots \subseteq V_i\subseteq\ldots\] such that $V=\bigcup_{i}V_i$.  Then $V$ is isomorphic as a vector space to the graded vector space $\bigoplus_i V_{[i]}$, where $V_{[0]}=V_0$, and $V_{[i+1]}=V_{i+1}/V_i$.

If $W$ is a subspace of $V$, then we have 
\[(W\cap V_0)\subseteq (W\cap V_1)\subseteq\ldots\subseteq (W\cap V_i)\subseteq\ldots,\]
with $W=\bigcup_i(W\cap V_i)$.  It follows that $W$ is isomorphic as a vector space to the graded vector space $\bigoplus_i W_{[i]}$, where $W_{[0]}\cong (W\cap V_0)$, and for each $i$, 
\[W_{[i+1]}\cong(W\cap V_{i+1})/(W\cap V_i)\cong((W\cap V_{i+1})+V_i)/V_i.\]
Thus $W_{[i+1]}$ is isomorphic to the subspace of $V_{[i+1]}$ given by
\[\{w+V_i\mid w\in W\cap V_{i+1}\}.\]
We can think of this as the collection of ``leading terms'' of elements of $W\cap V_{i+1}$.

Now consider the algebra $T(V)$ for some graded vector space $V=\bigoplus_i V_{[i]}$.  Every element of $T(V)$ can be expressed as a linear combination of elements of the form $v_1v_2\ldots v_n$, where each $v_t$ is homogeneous---that is, $v_t\in V_{[k]}$ for some $k$.  For each homogeneous $v\in V_{[k]}$, we write $|v|=k$. Thus the grading of $V$ induces a filtration 
\[T(V)_0\subseteq T(V)_1\subseteq\ldots\subseteq T(V)_i\subseteq\ldots\]
on $T(V)$, with
\[T(V)_{i}=span\{v_1\ldots v_n : |v_1|+\ldots+|v_n|\leq i\}.\]
This filtration induces a grading $T(V)=\bigoplus_i T(V)_{[i]}$, where
\[T(V)_{[i]}=span\{v_1\ldots v_n : |v_1|+\ldots|v_n|=i\}.\]

Let $I$ be an ideal of $T(V)$, and consider the algebra $A=T(V)/I$.  The filtration on $T(V)$ induces a filtration
\[A_0\subseteq A_1\subseteq \ldots A_i\subseteq \ldots\]
on $A$. for $a\in A$ with $a\in A_i\setminus A_{i-1}$, we write $|a|=i$. If the ideal $I$ is homogeneous with respect to the grading on $T(V)$, then $A$ inherits this grading.  Otherwise, we can consider the associated graded algebra, denoted $gr A$ and given by
\[gr A=\bigoplus A_{[i]},\]
where $A_{[0]}=A_0$, $A_{[i+1]}=A_{i+1}/A_i$, and where multiplication is given by
\[(x+A_m)(y+A_n)=xy+A_{m+n}.\]
The associated graded algebra $gr A$ is isomorphic to $A$ as a vector space, but not necessarily as an algebra.

For $A=T(V)/I$, we can understand the structure of $gr A$ by considering the graded structure of the ideal $I$.  As a vector space, $I$ is isomorphic to the vector space $gr I=\bigoplus (gr I)_{[i]},$ where $(gr I)_{[0]}=I\cap T(V)_0$, and
\[(gr I)_{[i+1]}=\{w+T(V)_i : w\in I\cap T(V)_{i+1}\}.\]
If we think of $T(V)$ as the direct sum $\bigoplus T(V)_{[i]}$, then we can consider each $(gr I)_{[i]}$ as a subset of $T(V)_{[i]}$, and we can think of $gr I$ as the collection of sums of leading terms of elements of $I$.  This is a graded ideal in the graded algebra $T(V)$.  From \cite{KozulPaper}, we have the following result:

\begin{lemma}
Let $V$ be a graded vector space and $I$ an ideal in $T(V)$.  Then 
\[gr(T(V)/(I))\cong T(V)/(gr(I)).\]
\end{lemma}

In the case of the universal labeling algebra $A(\Gamma)$ for a layered graph $\Gamma$ with unique minimal vertex $*$, we will use the notation $gr A(\Gamma)$ to denote the associated graded algebra with respect to the vertex filtration on $A(\Gamma)$.  In the sections that follow, we will see a particularly nice presentation of the algebra $gr A(\Gamma)$, and we will use the information that we gain from this presentation to construct a related algebra $B(\Gamma)$ for a particular class of layered graphs.

\subsection{Bases For $A(\Gamma)$ and $gr A(\Gamma)$}

Here we begin by recalling the basis for $A(\Gamma)$ given by Gelfand, Retakh, Serconek, and Wilson in \cite{AlgebrasAssocToDirGraphs}. Let $\Gamma=(V,E)$ be a layered graph.  For every ordered pair $(v,k)$ with $v\in V$ and $0\leq k\leq |v|$, we define $\tilde{e}(v,k)=\tilde{e}(\pi_v,k)$.  We define $\mathbb{B}_1$ to be the collection of all sequences of ordered pairs
\[\mathbb{b}=((b_1,k_1),\ldots,(b_n,k_n)),\]
and for each such sequence, we define
\[\tilde{e}(\mathbb{b})=\tilde{e}(b_1,k_1)\ldots\tilde{e}(b_n,k_n).\]
Given $(v,k)$ and $(v^\prime,k^\prime)$ with $v,v^\prime\in V$,
 $0\leq k\leq|v|$, and $0\leq k^\prime,\leq |v^\prime|$, we say that $(v,k)$ ``covers'' $(v^\prime,k^\prime)$, and write $(v,k)\forces(v^\prime,k^\prime)$ if $v>v^\prime$ and $|v|-|v^\prime|=k$.  That is to say, $(v,k)\models(v^\prime,k^\prime)$ if and only if there exists a path of length $k$ from $v$ to $v^\prime$.  Notice that this relation does not depend on $k^\prime$.

Define $\mathbb{B}$ to be the collection of sequences
\[\mathbb{b}=((b_1,k_1),\ldots,(b_n,k_n))\]
of ordered pairs such that for any $1\leq i<n$, $(b_i,k_i)\not\forces(b_{i+1},k_{i+1})$.  In \cite{AlgebrasAssocToDirGraphs}, Gelfand, Retakh, Serconek, and Wilson show that $\{\tilde{e}(\mathbb{b}):\mathbb{b}\in \mathbb{B})$ is a basis for $A(\Gamma)$.

In the same paper, a basis for $gr A(\Gamma)$ is also constructed as follows:  For each $\mathbb{b}\in\mathbb{B}_1$, with $|\tilde{e}(\mathbb{b})|=i$ with respect to the vertex-filtration, define $\overline{e}(\mathbb{b})=\tilde{e}(\mathbb{b})+A(\Gamma)_{i-1}$ in $gr A(\Gamma)$.  Then for any distinct $\mathbb{b}$ and $\mathbb{b}^\prime$ in $\mathbb{B}$, we have $\overline{e}(\mathbb{b})\neq\overline{e}(\mathbb{b}^\prime)$, and the set $\{\overline{e}(\mathbb{b}):\mathbb{b}\in \mathbb{B}\}$ is a basis for $gr A(\Gamma)$.

With a little extra notation, we can describe this basis in another way.  

\begin{defn}For every ordered pair $(v,k)$ with $v\in V$, $0\leq k\leq |v|$ let $\pi_v$ denote the vertex path $(v_1,v_2,\ldots,v_n)$. We define the monomial
\[m(v,k)=v_1v_2\ldots v_k\in T(V_+).\]
For each $\mathbb{b}=((b_1,k_1),\ldots,(b_n,k_n))\in \mathbb{B}_1$, we define
\[m(\mathbb{b})=m(b_1,k_1)\ldots m(b_n,k_n).\]
\end{defn}

\begin{prop}
Let $\Gamma$ be a layered graph, and let $\phi$ be the quotient map from $T(V_+)$ to $T(V_+)/gr R_V\cong gr A(\Gamma)$.  Then for all $\mathbb{b}\in\mathbb{B}_1$, we have $\phi(m(\mathbb{b}))=\overline e(\mathbb{b})$.
\end{prop}

\begin{proof}
We know that for $(v,k)$ with $v\in V$, $0\leq k\leq |v|$, and with vertex path $\pi_v=(v_1,\ldots,v_n),$ we define ${e}(v,k)$ by
\[(t-(v_1,v_2))(t-(v_2,v_3))\ldots(t-(v_{n-1},v_n))=\sum_{i=0}^{|v|}e(v,i)t^i.\]
Thus we have
\[\tilde{e}(v,k)=\hspace{-12pt}\sum_{1\leq i_1<i_2<\ldots i_k<n}\hspace{-12pt}(-1)^k\left(v_{i_1}-v_{(i_1+1)}\right)\ldots\left(v_{i_k}-v_{(i_k+1)}\right),\]
and the highest-order term of $\tilde{e}(v,k)$ with respect to the vertex filtration is the monomial $v_1v_2\ldots v_k$, or $m(v,k)$.  It follows that $\overline{e}(v,k)=\phi(m(v,k))$, and so by extension we have $\overline{e}(\mathbb{b})=\phi(m(\mathbb{b}))$ for all $\mathbb{b}\in\mathbb{B}_1$.
\end{proof}

\begin{cor}
\label{grbasis}
For any distinct $\mathbb{b}$ and $\mathbb{b}^\prime$ in $\mathbb{B}$, we have $\phi(m(\mathbb{b}))\neq\phi(m(\mathbb{b}^\prime))$, and the set $\{\phi(m(\mathbb{b})) : \mathbb{b}\in \mathbb{B}\}$ is a basis for $gr A(\Gamma)$.
\end{cor}

\subsection{Presentation of $gr A(\Gamma)$}

We can use the basis we have just defined to find a nice presentation of $gr A(\Gamma)$.  Let $R_{gr}$ be the ideal of $T(V_+)$ generated by all elements of the form $v_0v_1\ldots v_n-w_0w_1\ldots w_n$, where $(v_0,v_1,\ldots,v_n)$ and $(w_0,w_1,\ldots w_n)$ are vertex paths with $v_0=w_0$.

\begin{theorem}
\label{PresentationOfGrA}
$gr A(\Gamma)\cong T(V_+)/R_{gr}$.
\end{theorem}

To prove this, we will need to show that $R_{gr}=gr R_V$.  Showing that $R_{gr}\subseteq gr R_V$ is fairly straightforward.

\begin{prop}
Let $(v_0,v_1,\ldots,v_n)$ and $(w_0,w_1,\ldots,w_n)$ be vertex paths with $v_0=w_0$.  Then
\[v_0v_1\ldots v_n-w_1w_2\ldots w_n\in gr R\]
\end{prop}

\begin{proof}
Let $(v_0,v_1,\ldots,v_n)$ and $(w_0,w_1,\ldots,w_n)$ be vertex paths with $v_0=w_0$.  Since every vertex in $V_+$ has nonzero out-degree, we can extend each of these paths to a path that begins at $v_0$ and ends at $*$.  Thus we can define vertex paths
\[\pi_v=(v_0,v_1,\ldots, v_{|v_0|})\]
and
\[\pi_w=(w_0,w_1,\ldots,w_{|w_0|})\]
so that $v_{|v_0|}=w_{|w_0|}=*$.  Since $\pi_v\sim\pi_w$, we have
\[\tilde{e}(\pi_v,n+1)-\tilde{e}(\pi_w,n+1)\in R_V\]
We have
\[\tilde{e}(\pi_v,n+1)=\hspace{-24pt}\sum_{0\leq i_0<i_2<\ldots<i_{n}\leq|v_0|}\hspace{-24pt}
\left(v_{i_0}-v_{i_0+1}\right)\left(v_{i_1}-v_{i_1+1}\right)\ldots\left(v_{i_{n}}-v_{i_{n}+1}\right)\]
The leading term of this expression is $v_0v_1\ldots v_n$.  Similarly, the leading term of $\tilde{e}(\pi_w,n+1)$ is $w_0w_1\ldots w_n$.  It follows that
\[v_0v_1\ldots v_n-w_0w_1\ldots w_n\in gr R.\]
\end{proof}

\begin{cor}
\label{Rinclusion}
$R_{gr}\subseteq gr R_V$.
\end{cor}

Corollary~\ref{Rinclusion} implies that the quotient map $\phi:T(V_+)\rightarrow T(V_+)/gr R_V$ factors uniquely through $T(V_+)/R_{gr}$, giving us
\[T(V_+)\overset{\phi^\prime}{\longrightarrow} T(V_+)/R_{gr}\overset{\phi^{\prime\prime}}{\longrightarrow} T(V_+)/gr R.\]
We will use this notation for these three maps in the discussion that follows.  We will also borrow some terminology from \cite{QnPaper}.

\begin{defn}
Let $a=v_1,\ldots v_l$ be a monomial in $T(V_+)$.  We define $\textbf{s}^a$, the \textit{skeleton} of $a$, to be the sequence of integers $(s^a_1,\ldots,s^a_t)$ satisfying 
\begin{itemize}
\item[(i)] $s^a_1=1.$
\item[(ii)] If $s^a_k<l+1$, then
\[s^a_{k+1}=\min\big(\left\{j>s^a_k : v_j\neq h(e_{v_{j-1}})\right\}\cup\{l+1\}\big).\]
\item[(iii)] $t=\min\{i: s^a_i=l+1\}$.
\end{itemize}
\end{defn}

\begin{defn}
Let $a$ be a monomial in $T(V_+)$, and let $\textbf{s}^a=\left(s^a_1,s^a_2,\ldots,s^a_t\right)$. We define $\mathbb{b}_a\in\mathbb{B}_1$ to be
\[\left(\left(v_{s^a_1},s^a_2-s^a_1\right),\left(v_{s^a_2},s^a_3-s^a_2\right),\ldots,\left(v_{s^a_{t-1}},s^a_t-s^a_{t-1}\right)\right)\]
\end{defn}
Notice that for each monomial $a$, we have $m(\mathbb{b}_a)=a$.

\begin{prop}
\label{badmonomials}
Let $\Gamma=(V,E)$ be a layered graph, and let $a$ be a monomial in $T(V_+)$.  Then there exists a monomial $a^\prime\in T(V_+)$ such that $\phi^\prime(a)=\phi^\prime(a^\prime)$, and $\mathbb{b}_{a^\prime}\in \mathbb{B}.$
\end{prop}

\begin{proof}
Let $\Gamma=(V,E)$ be a uniform layered graph, and let $a=v_1,\ldots,v_l$ be a monomial in $T(V_+)$ with skeleton $\textbf{s}^a=(s^a_1,\ldots,s^a_t)$.  If $\mathbb{b}_a\notin\mathbb{B}$, then there exists $1<i<t$ such that
\[(v_{s^a_{i-1}},s^a_i-s^a_{i-1})\models(v_{s^a_i},s^a_{i+1}-s^a_i).\]
We define
\[r_a=\min\left(\left\{i : (v_{s^a_{i-1}},s^a_i-s^a_{i-1})\models(v_{s^a_i},s^a_{i+1}-s^a_i)\right\}\cup\{t\}\right)\]
We will induct on $l-s^a_{r_a}+1$. If $l-s^a_{r_a}+1=0$, then $\mathbb{b}_a\in \mathbb{B}$.  Now assume that $l-s^a_{r_a}>0$, and that the result holds for all monomials $a^\prime$ with $s^{a^\prime}_{r_{a^\prime}}>s^a_{r_a}$.

We have
\[a=v_1\ldots v_{\left(s^a_{(r_a-1)}-1\right)} m\left(v_{s^a_{\left(r_a-1\right)}},s^a_{r_a}-s^a_{\left(r_a-1\right)}\right)v_{s^a_{r_a}}\ldots v_l\]

Since we have
\[\left(v_{s^a_{\left(r_a-1\right)}},s^a_{r_a}-s^a_{\left(r_a-1\right)}\right)\models\left(v_{s^a_{r_a}},s^a_{{r_a}+1}-s^a_{r_a}\right),\]
we know that there exists a vertex path $(w_1,\ldots,w_h)$ with $w_1=v_{s^a_{(r_a-1)}}$ and $w_h~=~v_{s^a_{r_a}}$.

By the definition of $R_{gr}$, we have
\[\phi^\prime\left(m\left(v_{s^a_{(r_a-1)}},s^a_{r_a}-s^a_{(r_a-1)}\right)\right)=\phi^\prime\left(w_1\ldots w_{h-1}\right),\]
and 
\[\phi^\prime(w_1,\ldots,w_h)=\phi^\prime\left(m\left(v_{s^a_{(r_a-1)}},s^a_{r_a}-s^a_{(r_a-1)}+1\right)\right).\]
It follows that
\[\phi^\prime\left(m\left(v_{s^a_{(r_a-1)}},s^a_{r_a}-s^a_{(r_a-1)}\right)v_{s^a_{r_a}}\right)=\phi^\prime\left(m\left(v_{s^a_{(r_a-1)}},s^a_{r_a}-s^a_{(r_a-1)}+1\right)\right).\]

Thus we have $\phi^\prime(a)=\phi^\prime(a^\prime)$, where
\[a^\prime=v_1\ldots v_{(s^a_{r_a-1}-1)} m\left(v_{s^a_{(r_a-1)}},s^a_{r_a}-s^a_{(r_a-1)}+1\right)v_{({s^a_{r_a}}+1)}\ldots v_l\]

Say $a^\prime=v_1^\prime,\ldots v_l^\prime$, and $\textbf{s}^{a^\prime}=\left(s^{a^\prime}_1,\ldots,s^{a^\prime}_{t^\prime}\right)$.

We have 
\begin{itemize}
\item[(i)] $v_i=v_i^\prime$ for $1\leq i\leq {s^a_{r_a}}-1$.
\item[(ii)] $s^a_j=s^{a^\prime}_j$ for $j<r_a-1$.  
\item[(iii)] $s^{a^\prime}_{r_a}\geq s^a_{r_a}+1.$
\end{itemize}
Together, $(i)$ and $(ii)$ tell us that
\[\left(v_{s^{a^\prime}_{i-1}}^\prime,s^{a^\prime}_{i}-s^{a^\prime}_{i-1}\right)\not\models \left(v_{s^{a^\prime}_{i}}^\prime,s^{a^\prime}_{i+1}-s^{a^\prime}_i\right)\]

Thus by the inductive hypothesis there exists a monomial $a^{\prime\prime}\in T(V_+)$ with $\phi^\prime(a^\prime)=\phi^\prime(a^{\prime\prime})$ and $\mathbb{b}_{a^{\prime\prime}}\in\mathbb{B}$.  We have $\phi^\prime(a)=\phi^\prime(a^{\prime\prime})$, and so our proof is complete.
\end{proof}

\begin{prop}
For any distinct $\mathbb{b}$ and $\mathbb{b}^\prime$ in $\mathbb{B}$, we have $\phi^\prime(m(\mathbb{b}))\neq\phi^\prime(m(\mathbb{b}^\prime))$, and the set $\{\phi^\prime(m(\mathbb{b})) : \mathbb{b}\in \mathbb{B}\}$ is a basis for $T(V_+)/R_{gr}$.
\end{prop}

\begin{proof}
For any $\mathbb{b}\in \mathbb{B}$, we have $\phi(m(\mathbb{b}))=\phi^{\prime\prime}\circ\phi^\prime(m(\mathbb{b}))$.  Thus the elements of 
\[\{\phi^\prime(m(\mathbb{b})) : \mathbb{b}\in \mathbb{B}\}\]
are distinct and linearly independent by Corollary~\ref{grbasis}.  Since the set
\[\{a\in T(V_+) : a\text{ a monomial}\}\]
spans $T(V_+)$, it follows that the set
\[\{\phi^\prime(a) : a\text{ a monomial}\}\]
spans $T(V_+)/R_{gr}$.  Thus the set 
\[\{\phi^\prime(m(\mathbb{b})) : \mathbb{b}\in \mathbb{B}\}\]
spans $T(V_+)/R_{gr}$ as a consequence of Proposition~\ref{badmonomials}.  This gives us our result.
\end{proof}

\begin{cor}
$gr R_V\subseteq R_{gr}$
\end{cor}

\begin{proof}
For any $\mathbb{b}\in\mathbb{B}$, we have
\[\phi^{\prime\prime}\left(\phi^\prime\left(m(\mathbb{b})\right)\right)=\phi(m(\mathbb{b})).\]
Since $\phi^{\prime\prime}$ maps a basis of $T(V_+)/R_{gr}$ bijectively onto a basis of $gr A(\Gamma)$, it follows that $\phi^{\prime\prime}$ is a bijection, and thus an isomorphism.  The result follows.
\end{proof}

This final corollary gives us Theorem~\ref{PresentationOfGrA}.

\section{Quadratic Algebras and Uniform Layered Graphs}
\label{section:Uniform}

We begin this section by recalling the following definitions:

\begin{defn}
An algebra $A$ is called \textbf{quadratic} if $A\cong T(W)/\langle R\rangle$, where $W$ is a finite-dimensional vector space, and $\langle R\rangle $ is an ideal generated by some subspace $R$ of $W\otimes W$.
\end{defn}

\begin{defn}
For each $v\in V_+$, we define an equivalence relation $\sim_v$ on $S(v)$ to be the transitive closure of the relation $\approx$ on $S(v)$ given by $w\approx_v u$ whenever $S(w)\cap S(u)\neq\emptyset$.  A graph $\Gamma$ is said to be a \textbf{uniform layered graph} if for any $v\in V_{\geq2}$, all elements of $S(v)$ are equivalent under $\sim_v$.
\end{defn}

In \cite{KozulPaper}, Retakh, Serconek, and Wislon prove that $A(\Gamma)$ is quadratic if $\Gamma$ is uniform.  In \cite{Shelton}, Shelton proves that $gr A(\Gamma)$ is quadratic if and only if $\Gamma$ is uniform.  In this case, $gr A(\Gamma)$ has a particularly nice presentation.

\begin{prop}
\label{gruniform}
If $\Gamma$ is a uniform layered graph, then
\[gr A(\Gamma)\cong T(V_+)/\langle v(u-w) : u,w\in S(v)\rangle\]
\end{prop}

To prove this, we will find it useful to consider the following equivalent conditions for uniformity:

\begin{prop}
\label{downup}
Let $\Gamma$ be a layered graph.  Then the following conditions are equivalent:

 \begin{itemize}
 \item[(i)] $\Gamma$ is uniform.
 \item[(ii)] For any $v,x,x^\prime$ with $x\lessdot v$ and $x^\prime\lessdot v$, there exist seqences of vertices $x_0,\ldots,x_s$ and $y_1,\ldots, y_s$ such that
\begin{itemize}
\item[(a)] $x=x_0$ and $x^\prime=x_s$.
\item[(b)] For all $i$ such that $0\leq i\leq s$ we have $x_i\lessdot v$.
\item[(c)] For all $i$ such that $1\leq i<s,$ we have $y_i\lessdot x_{i-1}$ and $y_i\lessdot x_{i}$.
\end{itemize}
\item[(iii)] for any two vertex paths $(v_1,v_2,\ldots,v_n)$ and $(w_1,w_2,\ldots,w_n)$ with $v_1=w_1$ and $v_n=w_n=*$, there exists a sequence of vertex paths $\pi_1,\pi_2,\ldots,\pi_k$, each path beginning at $v_1$ and ending at $*$, such that for $1\leq i<k$, the vertex paths $\pi_i$ and $\pi_{i+1}$ differ by at most one vertex.
\end{itemize}
\end{prop}

\begin{figure}[h]
\caption{Two vertices $u$ and $w$ below $v$, connected by a down-up sequence as described in condition (ii) of Proposition~\ref{downup}.}
\begin{center}
\includegraphics[scale=.2]{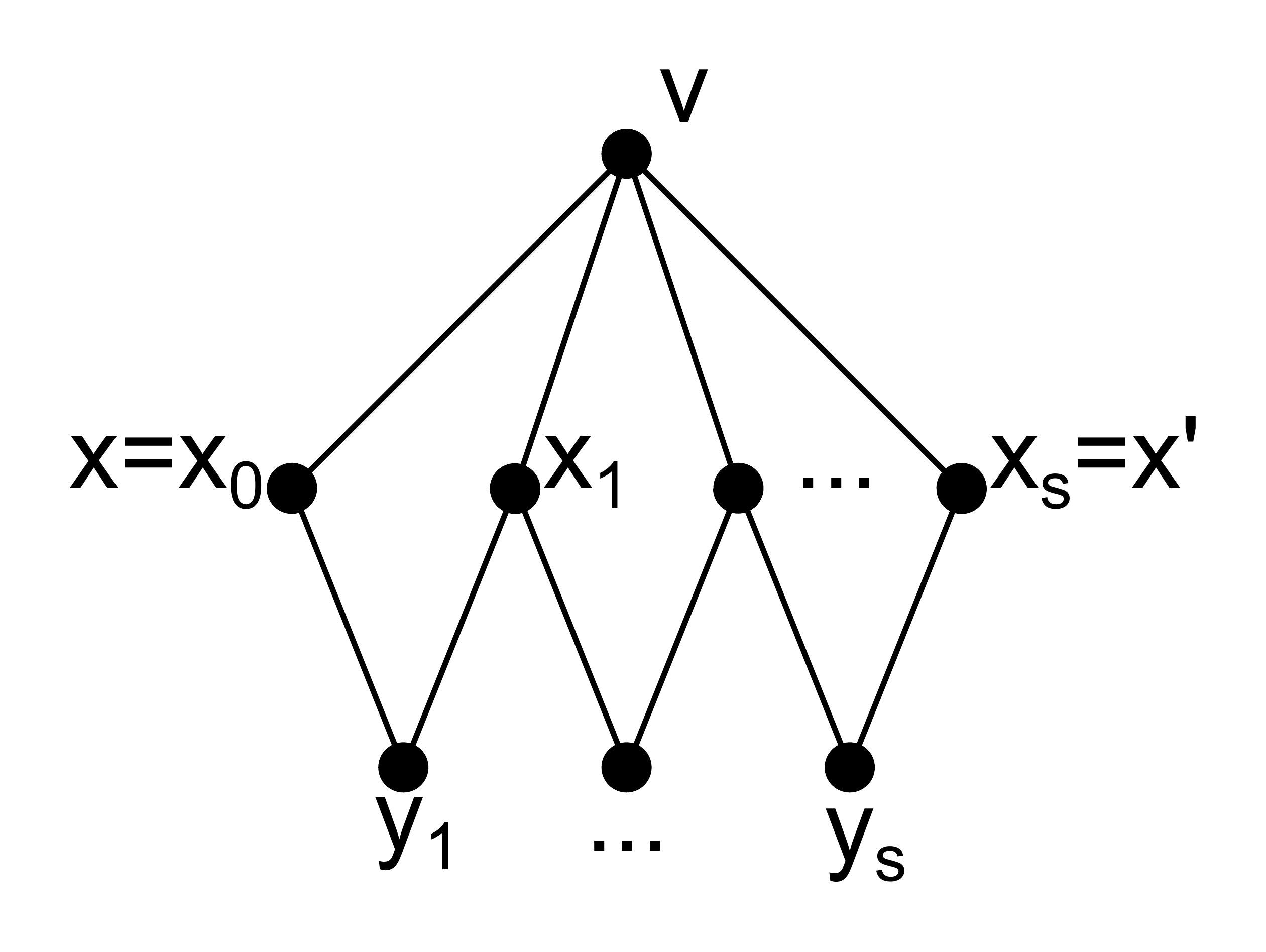}
\end{center}
\end{figure}

\begin{proof}

The fact that (i) and (ii) are equivalent follows directly from the definition of the relation $\sim_v$. A pair of sequences of vertices that satisfy (a)-(c) is sometimes referred to as a ``down-up sequence.''  We will proceed by showing that conditions (ii) and (iii) are equivalent.

To show that (ii) implies (iii), we will induct on $n$, the number of vertices in the paths.  Condition (iii) clearly holds if $n=1,2$, establishing our base case.  Now we let $n>2$, and assume that (iii) holds for all paths with fewer than $n$ vertices.  Let $(v_1,v_2,\ldots,v_n)$ and $(w_1,w_2,\ldots,w_n)$ be paths satisfying $v_1=w_1$ and $v_n=w_n=*$.  Condition (ii) states that there exists a down-up sequence, consisting of two sequences of vertices $x_0,\ldots,x_s$ and $y_1,\ldots,y_s$ satisfying
\begin{itemize}
\item[(a)] $v_2=x_0$ and $w_2=x_s$.
\item[(b)] For all $i$ such that $0\leq i\leq s$ we have $x_i\lessdot v$.
\item[(c)] For all $i$ such that $1\leq i<s,$ we have $y_i\lessdot x_{i-1}$ and $y_i\lessdot x_{i}$.
\end{itemize}
Let $y_0=v_3$, and let $y_{s+1}=w_3$.  For $0\leq i\leq s$, our induction hypothesis tells us that there exists a sequence of vertex paths $\pi_1^i, \pi_2^i,\ldots,\pi_{k_i}^i$ such that
\begin{itemize}
\item[($\alpha$)] For $0\leq i\leq s$, $\pi_1^i=x_i\wedge\pi_{y_i}$ and $\pi_{k_i}^i=x_i\wedge\pi_{y_{i+1}}$.
\item[($\beta$)] For $1\leq j<k_i$, the paths $\pi_j^i$ and $\pi_{j+1}^i$ differ by only one vertex.
\end{itemize}
The path-sequence that we wish to obtain is given by
\[v_1\wedge\pi_1^1,(v_1,x_1)\wedge\pi_2^1,\ldots,(v_1,x_1)\wedge\pi_{k_1}^1,(v_1,x_2)\wedge\pi_1^2,
\ldots, (v_1,x_2)\wedge\pi_{k_2}^2,\ldots\]\[\ldots,(v_1,x_s)\wedge\pi_1^s,\ldots,(v_1,x_s)\wedge\pi_{k_s}^s=(w_1,w_2,\ldots,w_n).\]

To prove that (iii) implies (ii), suppose that for any pair of vertex paths $(v_1,\ldots,v_n)$ and $(w_1,\ldots,w_n)$ satisfying $v_1=w_1$ and $v_n=w_n=*$, we have a sequence of vertex paths $\pi_1,\ldots\pi_k$ beginning at $v_1$ and ending at $*$ such that for $1\leq i<k$, the vertex paths $\pi_i$ and $\pi_{i+1}$ differ at at most one vertex.  

Let $v, x,$ and $x^\prime$ be vertices in $\Gamma$ such that $x\lessdot v$ and $x^\prime\lessdot v$.  Then $v\wedge \pi_x$ and $v\wedge \pi_{x^\prime}$ are paths which both start at $v$ and end at $*$.  Thus there exists a sequence of paths $\pi_1,\ldots,\pi_k$ beginning at $v$ and ending at $*$, and differing in each step by at most one vertex, such that $\pi_1=v\wedge \pi_x$ and $\pi_k=v\wedge \pi_{x^\prime}$.  For $1\leq i\leq k$, let $x_i$ be the second vertex on path $\pi_i$, and let $y_i$ be the third vertex on path $\pi_i$.  Then the sequences $x_1,\ldots,x_k$ and $y_2,\ldots, y_{k-1}$ satisfy (a)-(c) from condition (ii), and thus constitute a down-up sequence.
\end{proof}

We will also find the following corollary useful.

\begin{cor}
\label{UniformCor}
Let $\Gamma$ be a uniform layered graph with unique minimal vertex $*$.  Then for any two vertex paths $(v_1,v_2,\ldots,v_n)$ and $(w_1,w_2,\ldots,w_n)$ with $v_1=w_1$, there exists a sequence of vertex paths $\pi_1,\pi_2,\ldots,\pi_k$, each path beginning at $v_1$, such that for $1\leq i<k$, the vertex paths $\pi_i$ and $\pi_{i+1}$ differ by at most one vertex.
\end{cor}

\begin{proof}[Proof of Proposition~\ref{gruniform}]
We know that $gr A(\Gamma)\cong T(V_+)/ gr R_V$, and that $gr R_V$ is generated by all elements of the form 
\[v_0\ldots v_n-w_0\ldots w_n,\]
where $(v_0,\ldots,v_n)$ and $(w_0,\ldots,w_n)$ are vertex paths with $v_0=w_0$.  Thus
\[\{v(u-w) : u,w\in S(v)\}\subseteq gr R_V,\]
so the map $\phi: T(V_+)\rightarrow T(V_+)/gr R_V$ factors through 
\[T(V_+)/\langle v(u-w) : u,w\in S(v)\rangle,\]
giving us
\[T(V_+)\overset{\phi^\prime}{\longrightarrow} T(V_+)/\langle v(u-w) : u,w\in S(v)\rangle\overset{\phi^{\prime\prime}}{\longrightarrow} T(V_+)/gr R_V.\]
It will suffice to show that $\ker(\phi)\subseteq\ker(\phi^\prime)$.  To do this, we must show that for any two vertex paths $(v_0,\ldots,v_n)$ and $(w_0,\ldots,w_n)$ with $v_0=w_0$, we have
\[\phi^\prime(v_0\ldots v_n)=\phi^\prime(w_0\ldots w_n)\]
Let $(v_0,\ldots,v_{i-1},v_i,v_{i+1},\ldots v_n)$ and $(v_0,\ldots,v_{i-1},v_i^\prime,v_{i+1},\ldots v_n)$ be two vertex paths differing by one vertex.  Since $v_i,v_i^\prime\in S(v_{i-1})$, we have
\[v_1\ldots v_{i-2}(v_{i-1}(v_i-v_i^\prime))v_{i+1}\ldots v_n\in \ker(\phi^\prime),\]
 and thus
\[\phi^\prime(v_1\ldots v_{i-1}v_iv_{i+1}\ldots v_n)=\phi^\prime(v_1\ldots v_{i-1}v_i^\prime v_{i+1}\ldots v_n),\in \ker(\phi^\prime).\]
This in combination with Corollary~\ref{UniformCor} and the uniformity of $\Gamma$ gives us our result.
\end{proof}

\subsection{$B(\Gamma)$, the Quadratic Dual of $gr A(\Gamma)$}

The ideal $R_V$ of $T(V_+)$ is often difficult to describe, making the algebra $A(\Gamma)$ hard to work with.  We will find it useful to consider a second algebra, defined as follows:
\begin{defn}
If $\Gamma$ is a layered graph, then $B(\Gamma)=T(V_+)/R_B$, where $R_B$ is the ideal generated by
\[\{vw: v,w\in V_+, v\ngtrdot w\}\cup\left\{v\sum_{v\gtrdot w}w : v\in V_+\right\}.\]
\end{defn}

\begin{figure}[h]
\caption{If this figure is a part of $\Gamma$, then in $B(\Gamma)$ we will have $va=0$, $ve=0$, and $v(b+c+d)=0$.  We will also have $v^2=0$, and $vw=0$ for any vertex $w$ not included in this figure.}
\label{singlevertexpic}
\begin{center}
\includegraphics[scale=.2]{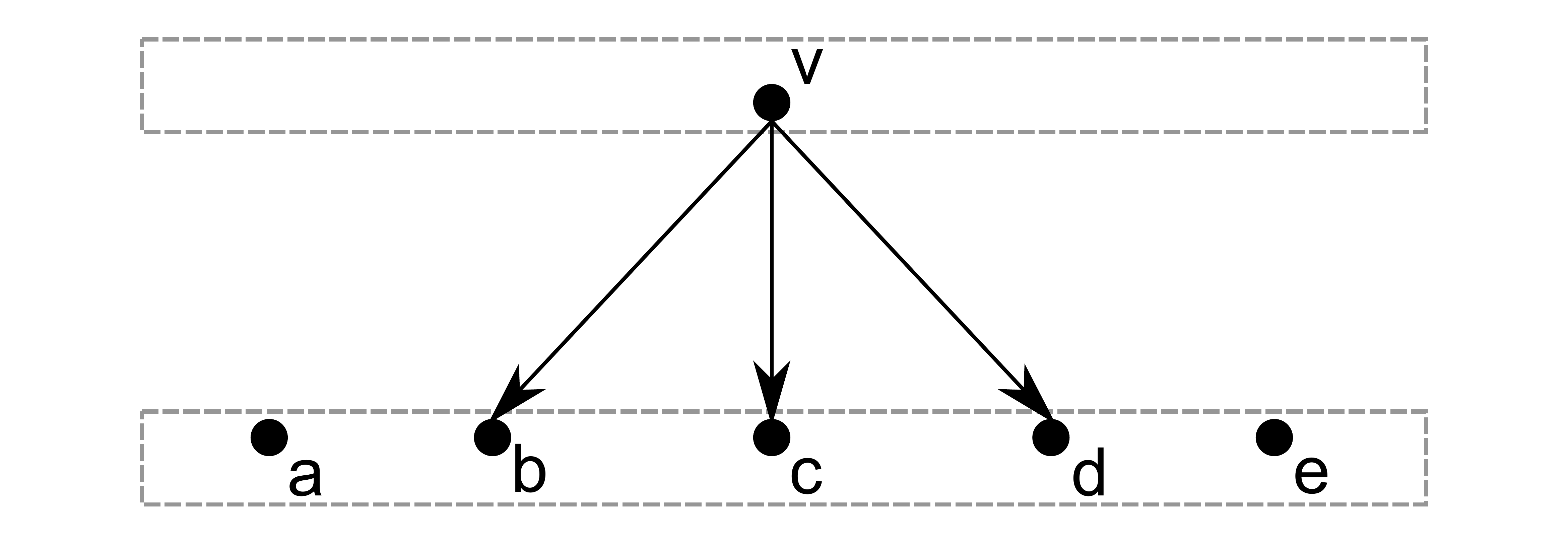}
\end{center}
\end{figure}

If $\Gamma$ is uniform, the algebra $B(\Gamma)$ can be calculated directly from $A(\Gamma)$.  To explain this relationship, we will recall the following definitions, as presented in \cite{QuadraticAlgebras}:

\begin{defn}
If $R$ is a subspace of $W\otimes W$, then $R^\perp$ is the ideal in $T(W^*)$ generated by the set of elements $x\in (W\otimes W)^*$ such that for all $y\in W\otimes W$, we have $\langle x,y\rangle=0$.
\end{defn}

\begin{defn}
Given a quadratic algebra $A\cong T(W)/R$, the quadratic dual $A^!$ is defined to be the algebra $T(W^*)/R^\perp$.
\end{defn}

In particular, when $\Gamma$ is uniform,  we can calculate $B(\Gamma)$ directly from $A(\Gamma)$ by way of the following proposition:

\begin{prop}
\label{bgamma}
Let $\Gamma$ be a uniform layered graph with unique minimal vertex.  Then $B(\Gamma)\cong (gr A(\Gamma))^!$.
\end{prop}

\begin{proof} By Proposition~\ref{gruniform}, whenever $\Gamma$ is uniform, $gr R_V$ is generated by 
\[\{v(w-u):v\in V_{\geq2}, u,w\in S(v)\},\]
$(gr R_V)^\perp$ is generated by the collection of $x^*y^*\in(V_+)^*\otimes(V_+)^*$ such that 
\[\langle x^*,v\rangle\langle y^*,w-u\rangle=0\]
for all $v\in V_{\geq2}$ and $u,w\in S(v)$.  In \cite{ColleensThesis}, Duffy shows that this is the ideal generated by
\[\{v^*w^*: v\ngtrdot w\}\cup\left\{v^*\sum_{v\gtrdot w}w^*\right\}.\]
A simple change of variables shows that this is isomorphic to the algebra $B(\Gamma)$.
\end{proof}

Just as $A(\Gamma)$ has a grading given by polynomial degree and a filtration given by the level of the vertices, $B(\Gamma)$ has a double grading given by 
\[B(\G)_{m,n}=\text{span}\left\{v_1\ldots v_m : \sum_{i=1}^m |v_i|=n\right\}.\]
Thus we can define an invariant $\sim_B$ which is analogous to the invariant $\sim_A$ as follows:

\begin{defn}
We say $\Gamma\sim_B\Gamma^\prime$ if and only if there exists an isomorphism $\phi:B(\Gamma)\rightarrow B(\Gamma^\prime)$ which preserves the double grading.
\end{defn}

We know from \cite{Shelton} that $\Gamma$ is a uniform layered graph if and only if $gr A(\Gamma)$ is quadratic.  In this case, $B(\Gamma)$ is isomorphic to its quadratic dual.  All of these transformations preserve both the degree grading and the vertex filtration, and so we have the following:

\begin{prop}
If $gr A(\Gamma)$ is quadratic, then \[(\Gamma\sim_A\Gamma^\prime) \Rightarrow (\Gamma\sim_B\Gamma^\prime).\]
\end{prop}

Thus $\sim_B$ gives us a coarser partition of the uniform layered graphs.  In particular, if the $\sim_B$ equivalence class of a uniform layered graph consists of a single isomorphism class, then so does the $\sim_A$ equivalence class.  If the $\sim_B$ equivalence class is small and easy to describe, we know that the $\sim_A$ equivalence class is contained in this small, easily-describable set of layered graphs. In sections \ref{chap:IsosOfB}-\ref{chap:UniquenessResults}, we will take advantage of this to show that several important classes of layered graphs are uniquely identifed by their algebra $A(\Gamma)$ via that $\sim_A$ relation.

\section{Isomorphisms of $B(\Gamma)$}
\label{chap:IsosOfB}
\subsection{Subspaces of $B(\G)$}

Here we will explore the structure of $B(\Gamma)$ considered with the double grading given by
\[B(\G)_{m,n}=\text{span}\left\{v_1\ldots v_m : \sum_{i=1}^m |v_i|=n\right\}.\]

\begin{defn}
For notational convenience, we define $B_n=B(\G)_{1,n}=\left\{\sum_{v\in V_n}\alpha_vv\right\}$, the linear span of the vertices in $V_n$.
\end{defn}

\begin{defn}
Given an element $a=\sum_{v\in V_n}\alpha_vv\in B_n,$  let $A_a~=~\{~v~\in~ V_n~:~\alpha_v~\neq ~0\}.$  Then we define
\begin{itemize}
\item[(i)]$S(a)=S(A_a)$
\item[(ii)]$\sim_a=\sim_{A_a}$
\item[(iii)]$k_a=k_{A_a}$
\item[(iv)]$k_a^a=k_{A_a}^{A_a}$
\item[(v)]$\mathscr{C}_a=\mathscr{C}_{A_a}$
\end{itemize}
\end{defn}

\subsection{The Subspaces $\kappa_a$}

\begin{defn}
  For each element $a\in B_n$, we define a map 
\[L_a: B_{n-1}\rightarrow B(\G)\]
such that $L_a(b)=ab$.  We define $\kappa_a$ be the kernel of the map $L_a$.
\end{defn}

If we consider the example from Figure \ref{singlevertexpic}, we see that $\kappa_v$ is generated by $\{a,e,b+c+d\}.$ In general, if we consider the generating relations for $R_B$, it is clear that for any $v\in V_n$, we have
\[\kappa_v=\text{span}\left(\{w:w\not\hspace{-4pt}\lessdot v\}\cup\left\{\sum_{w\lessdot v}w\right\}\right)\]
It follows that if $|S(v)|>1$, then for any vertex $w$, we have $w\in S(v)$ if and only if $w\notin \kappa_v$.  We can also obtain the following results about the structure of $\kappa_a$:
\begin{lemma}
\label{IntersectingKappas}
For any nonzero $a=\sum_{v\in V_n}\alpha_vv$ in $B_n$,
\[\kappa_a=\bigcap_{\alpha_v\neq0}\kappa_v\]
\end{lemma}

\begin{proof}
Clearly, if $b\in\bigcap_{\alpha_v\neq 0}\kappa_v$, then $b\in\kappa_a$, which implies that $\bigcap_{\alpha_v\neq 0}\kappa_v\subseteq\kappa_a$.

To obtain $\kappa_a\subseteq \bigcap_{\alpha_v\neq 0}\kappa_v$, we note that $B(\Gamma)$ can be considered as a direct sum of vector spaces
\[B(\Gamma)=\bigoplus_{v\in V_+}v B(\Gamma).\]
If $b\in \kappa_a$, then
\[\sum_{v\in V_n}\alpha_vvb=0,\]
and so we must have $vb=0$ for every $v$ such that $\alpha_v\neq 0$.  Thus we must have $b\in \bigcap_{\alpha_v\neq 0}\kappa_v$, and hence $\kappa_a\subseteq\bigcap_{\alpha_v\neq 0}\kappa_v$.
\end{proof}

\begin{lemma}
\label{StructureOfKappa}
For any $a=\sum\alpha_vv\in B_n$,
\[\kappa_a=\text{span}\left\{\sum_{w\in C}w : C\in\mathscr{C}_a\right\}\]
\end{lemma}

\begin{proof}
In light of Lemma~\ref{IntersectingKappas}, this statement reduces to showing that for any subset $A\subseteq V_n$,
\[\bigcap_{v\in A}\kappa_v=\text{span}\left\{\sum_{w\in C}w : C\in\mathscr{C}_A\right\}\]
We have
\begin{eqnarray*}
\bigcap_{v\in A}\kappa_v&=&\bigcap_{v\in A}\left(\text{span}\left(\{w:w\not\hspace{-4pt}\lessdot v\}\cup\left\{\sum_{w\lessdot v}w\right\}\right)\right)\\
&=&\bigcap_{v\in A}\left\{\left(\sum_{w\in V_{n-1}}\beta_ww \right): \beta_w=\beta_{w^\prime}\text{ if } \{w,w^\prime\}\subseteq S(v)\right\}\\
&=&\left\{\left(\sum_{w\in V_{n-1}}\beta_ww \right):\beta_w=\beta_{w^\prime}\text {if }\{w,w^\prime\}\subseteq S(v)\text{ for some }v\in A\right\}\\
&=&\left\{\left(\sum_{w\in V_{n-1}}\beta_ww \right):\beta_w=\beta_{w^\prime}\text{ if }\{w,w^\prime\}\subseteq C\in \mathscr{C}_A\right\}\\
&=&\text{span}\left\{\sum_{w\in C}w : C\in\mathscr{C}_A\right\}
\end{eqnarray*}
\end{proof}

Often it will be useful to be able to refer to $\kappa_A$ for a subset $A\subseteq V_n$:
\begin{defn}
\[\kappa_A=\bigcap_{v\in A}\kappa_v.\]
\end{defn}

Notice that $\kappa_A=\kappa_a$ for $a=\left(\sum_{v\in A}v\right)\in B_n$.  

From Lemma~\ref{StructureOfKappa}, it is easy to see that the following corollaries hold:

\begin{cor}
\label{SizeOfS}
For any $a\in B_n$, $\dim(\kappa_a)=k_a$, and for any $A\subseteq V_n,$ $\dim(\kappa_A)=k_A$.
\end{cor}

\begin{cor}
\label{OutDegree}
For any $a\in B_n$, $|S(a)|=|V_{n-1}|-k_a+k_a^a$, and $|S(A)|=|V_{n-1}|-k_A+k_A^A$
\end{cor}

In general, we cannot recover information about $k_a^a$ from $B(\G)$, but in the special case where $v$ is a vertex in $V_n$, we have $k_v^v=1$, and so $|S(v)|=|V_{n-1}|-k_v+1.$

We make the following observations about $\kappa_a$:

\begin{prop}
\label{Kappa}
Let $a=\sum\alpha_vv\in B_n$, and let $w$ be a vertex in $V_n$ such that $\alpha_w\neq 0$.  Then the following statements hold:
\begin{itemize}
\item[(i)]$\kappa_a\subseteq\kappa_w$
\item[(ii)]$k_a\leq k_w$
\item[(iii)]$S(w)\subseteq S(a)$
\end{itemize}
We have equality in $i)$ if and only if we have equality in $ii)$.  If $|S(w)|>1$, then we have equality in $iii)$ if and only if we have equality in $i)$ and $ii)$.  In the case where $|S(w)|=1$, then equality in $iii)$ implies equality in $i)$ and $ii)$, but the other direction of implication does not hold.
\end{prop}
\begin{proof}
All  this is obvious from Lemma~\ref{IntersectingKappas} and Lemma~\ref{StructureOfKappa}, except that $\kappa_a=\kappa_w$ implies $S(a)=S(w)$ in the case where $|S(w)|>1$.  Suppose $\kappa_a=\kappa_w$.  Notice that for $u\in V_{n-1}$ we have $u\in S(a)$ if and only if $u\notin\kappa_a$, and $u\in S(w)$ if and only if $u\notin \kappa_w$.  Thus $u\in S(a)$ if and only if $u\in S(w)$, so $S(a)=S(w)$.
\end{proof}

\subsection{Isomorphisms from $B(\G)$ to $B(\G^\prime)$}

Let $\G=(V,E)$, and $\G^\prime(W,F)$ be uniform layered graphs.  Here we consider $B(\Gamma)=T(V_+)/R_B$ and $B(\G^\prime)=T(W_+)/R_B^\prime$.  A natural question to ask is which doubly graded algebra isomorphisms between  $T(V_+)$ and $T(W_+)$ induce isomorphisms between the doubly graded algebras $B(\Gamma)$ and $B(\Gamma^\prime)$.  The answer turns out to be fairly simple:

\begin{theorem}
\label{IsoConditions}
Let $\Gamma=(V,E)$ and $\Gamma^\prime=(W,F)$ be uniform layered graphs with algebras $B(\Gamma)=T(V_+)/R_B$ and $B(\G^\prime)=T(W_+)/R_B^\prime$ resepectively, and let
\[\phi:T(V_+)\rightarrow T(W_+)\]
be an isomorphism of doubly graded algebras.  Then $\phi$ induces a doubly graded algebra isomorphism from $B(\Gamma)$ to $B(\Gamma^\prime)$ if and only if $\kappa_{\phi(v)}=\phi(\kappa_v)$ for all $v\in V$.
\end{theorem}

For this result, we will need the following lemma:

\begin{lemma}
\label{IsoConditionsLemma}
 Let $\Gamma=(V,E)$ and $\G^\prime(W,F)$ be  uniform layered graphs, and let
\[\phi:T(V_+)\rightarrow T(W_+)\]
be an isomorphism of doubly graded algebras.  Then the following are equivalent:
\begin{itemize}
\item[(i)] For any $a\in B_n$, $\phi(\kappa_a)=\kappa_{\phi(a)}$.

\item[(ii)] For any $v\in V_n$, $\phi(\kappa_v)=\kappa_{\phi(v)}$.\end{itemize}
\end{lemma}

\begin{proof}
Clearly, we have (i) $\Rightarrow$ (ii).  To show that (ii) $\Rightarrow$ (i), let
\[a=\sum \alpha_v v.\]
We know from Lemma~\ref{IntersectingKappas} that
\[\kappa_a=\bigcap_{\alpha_v\neq 0}\kappa_v.\]
Since $\phi$ is injective, this implies that
\[\phi(\kappa_a)=\bigcap_{\alpha_v\neq 0}\phi(\kappa_v)=\bigcap_{\alpha_v\neq 0}\kappa_{\phi(v)}\]
Let $\phi(v)=\sum \beta_{v,w}w$.  Then
\[\kappa_{\phi(v)}=\bigcap_{\beta_{v,w}\neq 0}\kappa_w,\]
and
\[\phi(a)=\sum \left(\alpha_v\left(\sum \beta_{v,w}w\right)\right)=\sum \alpha_v\beta_{v,w}w,\]
so we have
\[\phi(\kappa_a)=\bigcap_{\substack{\alpha_v\neq 0\\ \beta_{v,w}\neq 0}}
\kappa_w=\kappa_{\phi(a)}\]
\end{proof}

\begin{proof}[Proof of Theorem~\ref{IsoConditions}]
For the purposes of this proof, we will let $B_n=B(\Gamma)_{1,n},$ and $B_n^\prime=B(\Gamma^\prime)_{1,n}$. We wish to show that $\phi(R_B)=R_B^\prime$ if and only if $\kappa_{\phi(v)}=\phi(\kappa_v)$ for all $v\in V_n$.  We have
\[R_B=\langle ab : a\in B_n, b\in \kappa_a\text{ or }b\notin B_{n-1}\rangle\]
and so
\[\phi(R_B)=\langle\phi(ab) : a\in B_n, b\in \kappa_a\text{ or }b\notin B_{n-1}\rangle.\]
We also have
\[R_B^\prime=\langle ab : a\in B^\prime_n, b\in \kappa_a\text{ or }b\notin B^\prime_{n-1}\rangle.\]

Suppose that $\kappa_{\phi(v)}=\phi(\kappa_v)$ for all $v\in V_+$.  Then by Lemma~\ref{IsoConditionsLemma}, we have $\kappa_{\phi(a)}=\phi(\kappa_a)$ for any $a\in B_n$.  This means that for any $a$, we have $\kappa_a=\phi\inv(\kappa_{\phi(a)})$, and thus $\kappa_{\phi\inv(a)}=\phi\inv(\kappa_a)$.

Let $ab$ be one of the generators of $R_B$.  Then $a\in V_n$ and either $b\in \kappa_a$ or $b\notin B_{n-1}$. If $b\in \kappa_a$, we have
\[\phi(ab)=\phi(a)\phi(b)\]
We know $\phi(b)\in\phi(\kappa_a)=\kappa_{\phi(a)}$, so $\phi(ab)\in R_B^\prime$.  Otherwise, $b\notin B_{n-1}$, and so we have $\phi(a)\in B^\prime_n$ and $\phi(b)\notin B^\prime_{n-1}$, so $\phi(a)\phi(b)=\phi(ab)$ is in $R_B^\prime$. It follows that $\phi(R_B)\subseteq R_B^\prime$.

$R_B^\prime$ is generated by elements of the form $ab$, for $a\in W_n$ and $b\in \kappa_a$ or $b\notin B^\prime_{n-1}$.  If $b\in \kappa_a$,  we have
\[ab=\phi(\phi\inv(a)\phi\inv(b))\]
We know that $\phi\inv(a)\in B_n$, and $\phi\inv(b)\in\phi\inv(\kappa_a)=\kappa_{\phi\inv(a)}.$  Thus $\phi\inv(a)\phi\inv(b)\in R_B$, and so $ab\in \phi(R_B)$.  If $b\notin B_{n-1}^\prime$, then we have $\phi\inv(a)\in B_n$ and $\phi\inv(b)\notin B_{n-1}$, and so $\phi\inv(ab)\in R_B$.  It follows that $ab\in \phi(R_B)$, and so $R_B^\prime\subseteq \phi(R_B)$.  This means that $R_B^\prime=\phi(R_B)$.

Conversely, suppose $\kappa_{\phi(v)}\neq\phi(\kappa_v)$ for some $v\in V_n$.  Then there must be some $a$ such that
\[a\in\left(\kappa_{\phi(v)}\setminus\phi(\kappa_v)\right)\cup\left(\phi(\kappa_v)\setminus\kappa_{\phi(v)}\right)\]

If $a\in (\kappa_{\phi(v)}\setminus\phi(\kappa_v))$, then we have $\phi\inv(a)\notin \kappa_v$, so $v\phi\inv(a)\notin R_B$.  However, we also have $a\in\kappa_{\phi(v)}$, so $\phi(v\phi\inv(a))=\phi(v)a\in R_B^\prime.$  Since $\phi$ is a bijection, this means $R_B^\prime\neq \phi(R_B)$.

If $a\in (\phi(\kappa_v)\setminus\kappa_{\phi(v)})$, then $v\phi\inv(a)\in R_B,$ but $\phi(v\phi\inv(a))=\phi(v)a\notin R_B$.  Since $\phi$ is a bijection, $R_B^\prime\neq \phi(R_B)$.

\end{proof}

\section{Upper Vertex-Like Bases}
\label{chap:UpperVLikeB}

We are interested in studying the equivalence classes of layered graphs under the relations $\sim_A$ and $\sim_B$.  As previously discussed, for uniform layered graphs $\Gamma$ and $\Gamma^\prime$, we have $\Gamma\sim_A\Gamma^\prime \Rightarrow \Gamma\sim_B\Gamma^\prime$, so for now we will focus on the relation $\sim_B$.  Suppose that for a particular graph $\Gamma$, we are given the doubly-graded algebra $B(\Gamma)$, but no additional information about the graph.  What information about $\Gamma$ can we recover?

If we could somehow identify the vertices in $B(\Gamma)$, we could recover quite a bit of information.  In particular, we would know $S(v)$ for every vertex $v$ of degree greater than 1.  Unfortunately, it is not always possible to recover the vertices from the algebra.  Consider the following graph $\Gamma$:

\begin{center}
\includegraphics[scale=.15]{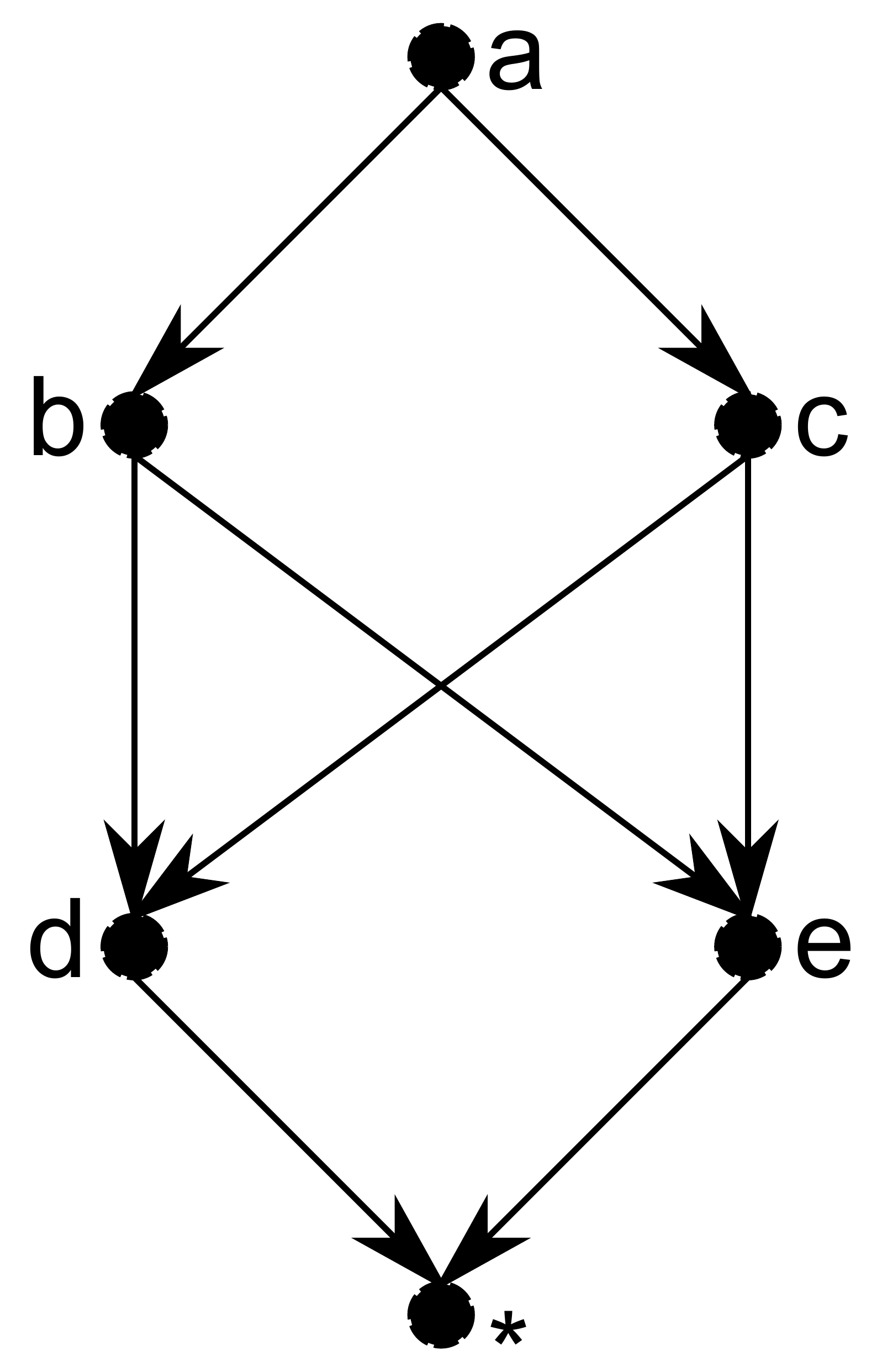}
\end{center}

Using Proposition~\ref{IsoConditions}, one can show that the algebra map $\phi$ given by
\[\phi(v)=\left\{
\begin{array}{ll}
b+c&\text{if }v=b\\
v&\text{ if }v\in V_+\setminus\{b\}
\end{array}
\right.\]
is an automorphism of $B(\Gamma)$ which does not fix the vertices.  Thus we cannot, in general, identify which elements of $B(\Gamma)$ are vertices.  However, it is possible to find a collection of algebra elements that in some sense ``act like'' vertices.  In this section, we will give a construction for these elements.

\subsection{Upper Vertex-Like Bases}

We will continue to use the notation $B_n$ for $B(\Gamma)_{1,n}$, the linear span of $V_n$ in $B(\Gamma)$.

\begin{defn}
If $\Gamma=(\bigcup_{i=0}^\infty V_i,\bigcup_{i=1}^\infty E_i)$ is a layered graph, then the \textbf{restriction of }$\Gamma$\textbf{ to }$n$ is
\[\Gamma|_n=\left(\bigcup_{i=0}^n V_i,\bigcup_{i=1}^n E_i\right)\]
\end{defn}

Notice that $\Gamma|_n$ is uniform whenever $\Gamma$ is uniform, and that in this case $B(\Gamma|_n)$ is the subalgebra of $B(\Gamma)$ generated by $\bigcup_{i=1}^nV_i$.

\begin{defn}
A basis $L$ for $B_n$ is called an \textbf{upper vertex-like} basis if there exists a doubly-graded algebra isomorphism
\[\phi: B(\Gamma|_n)\rightarrow B(\Gamma|_n)\]
such that $\phi$ fixes $\bigcup_{i=1}^{n-1}V_i$, and $\phi(V_n)=L$.
\end{defn}

In other words, an upper vertex-like basis is a collection of algebra elements that are indistinguishable from the vertices inside the subalgebra $B(\Gamma|_n)$.

\begin{prop}
\label{UpperVertexLikeBijection}
A basis $L$ for $B_n$ is upper vertex-like if and only if there exists a bijection $\phi:V_n\rightarrow L$ such that $\kappa_v=\kappa_{\phi(v)}$ for all $v\in V_n$.
\end{prop}

\begin{proof}
This follows from Theorem~\ref{IsoConditions}.  Define
\[\phi^\prime: T\left(\bigcup_{i=1}^nV_i\right)\rightarrow T\left(\bigcup_{i=1}^nV_i\right)\]
to be the graded algebra isomorphism given by
\[\phi^\prime(v)=\left\{
\begin{array}{ll}
v&\text{if } v\in \bigcup_{i=1}^{n-1}V_i\\
\phi(v)&\text{if } v\in V_n
\end{array}
\right.\]
Notice that $\phi^\prime$ fixes all of $B(\Gamma|_{n-1})$.  Since  $\kappa_v\subseteq B_{n-1}$, we have $\phi(\kappa_v)=\kappa_v$.  Thus $\kappa_v=\kappa_{\phi(v)}$ is exactly the condition we need for $\phi^\prime$ to induce an isomorphism on $B(\Gamma|_n)$.

Conversely, if $L$ is upper vertex-like, then there exists $\phi:B(\Gamma|_n)\rightarrow B(\G|_n)$ which fixes $\bigcup_{i=1}^{n-1}V_i$.  The restriction of $\phi$ to $V_n$ is a bijection from $V_n$ to $L$.  Theorem~\ref{IsoConditions} tells us that $\phi(\kappa_v)=\kappa_{\phi(v)}$.  Since $\phi$ fixes $\kappa_v$, we have our result.
\end{proof}

One advantage to an upper vertex-like basis for $B_n$ is that it allows us acces to the $\kappa_v$ subspaces for $v\in V_n$.  For convenience, here we will use the notation $\kappa(a)$ rather than $\kappa_a$.

\begin{prop}
Let $\G=\left(\bigcup_{i=0}^\infty V_i,E\right)$ be a layered graph, and let $L$ be an upper vertex-like basis for $B_n$.  Then there exists a bijection
\[\psi:V_n\rightarrow L\]
such that for any $A\subseteq V_n$, we have
\[\kappa\left(\sum_{v\in A}v\right)=\kappa\left(\sum_{v\in A}\psi(v)\right)\]
\end{prop}

\begin{proof}
Since $L$ is upper vertex-like, there exists an isomorphism
\[\phi:B(\G|_n)\rightarrow B(\G|_n)\]
which fixes $V_i$ for $1\leq i\leq n-1$, and takes $V_n$ to $L$.   Since $\phi(\kappa(a))=\kappa(\phi(a))$ for any $a\in B_n$, we have
\[\phi\left(\kappa\left(\sum_{v\in A}v\right)\right)=\kappa\left(\phi\left(\sum_{v\in A}v\right)\right)=\kappa\left(\sum_{v\in A}\phi(v)\right)=\kappa\left(\sum_{v\in A}\psi(v)\right),\]
and since $\phi$ fixes all elements of $B(\Gamma_{n-1})$, we have
\[\phi\left(\kappa\left(\sum_{v\in A}v\right)\right)=\kappa\left(\sum_{v\in A}v\right)\]
Thus we may take $\psi=\phi|_{V_n}$, and the result follows.
\end{proof}

This allows us to obtain information about the structure of the graph $\Gamma$ from the algebra $B(\Gamma)$.

\begin{example}
Consider the out-degree sequence of the graph $\Gamma$.  If $L=\{b_1,\ldots,b_{\dim(B_n)}\}$ is an upper vertex-like basis for $B_n$ and $\dim(\kappa_{b_i})=d$, then there exists $v\in V_n$ with $\dim(\kappa_v)=d$.  We have
\[|S(v)|=V_i-d+1,\]
and so $\Gamma$ must have a vertex in $V_n$ with out-degree $\dim(B_n)-d+1$.  In fact, since we can calculate $\dim(\kappa_{b_i})$ for all $1\leq i\leq \dim(B_n)$, these upper vertex-like bases allow us to calculate the degree sequence of the entire graph.  More precisely, this argument gives us
\begin{prop}
\label{OutDegree}
Let $\Gamma$ be a directed graph, and let $L$ be an upper vertex-like basis for $B_n\subseteq B(\Gamma)$.  Then multiset $\{|S(v)| : v\in V_n\}$ is equal to the multiset
\[\{\dim(B_n)-\dim(\kappa_b)+1 : b\in L\},\]
and thus can be calculated from $B(\Gamma)$.
\end{prop}
\end{example}

\begin{example}
We can also draw some inferences about the size of our $S(v)$ sets.  Suppose $\Gamma(V,E)$ is a layered graph, and $v,w\in V_n$ with $|S(v)\cap S(w)|\geq 1$.  Recall that
\[|S(A)|=V_{n-1}-k_A+k_A^A.\]
This gives us
\[|S(\{v,w\})|=V_{n-1}-k_{\{v,w\}}+k_{\{v,w\}}^{\{v,w\}},\]
Since $S(v)\cap S(w)\neq \emptyset$, we have $k_{\{v,w\}}^{\{v,w\}}=1$.  This in combination with the fact that
\[|S(v)|=V_{n-1}-k_v+1\]
and
\[|S(w)|=V_{n-1}-k_w+1\]
gives us
\[|S(v)\cap S(w)|=V_n+k_{v,w}-k_v-k_w+1.\]
If we have $S(v)\cap S(w)=\emptyset$, then $k_{\{v,w\}}^{\{v,w\}}=2,$ and so
\[|S(v)\cap S(w)|=1\]
This allows us to conclude the following:
\begin{prop}
\label{IntersectionSize}
Let $\Gamma$ be a directed graph, and let $L$ be an upper vertex-like basis for $B_n\subseteq B(\Gamma)$.  If there exist $b_1$ and $b_2$ in $B_n$ satisfying
\[V_{n-1}+k_{\{b_1,b_2\}}-k_{b_1}-k_{b_2}+1=m\geq 2,\]
then there exist $v,w\in V_n$ with $|S(v)\cap S(w)|=m$.
\end{prop}
\end{example}

\subsection{Constructing an Upper Vertex-Like Basis}
Our goal in this section is to prove the following theorem, which shows that we can construct an upper vertex-like basis for $B_n$ using only the information given to us by the doubly-graded algebra $B(\Gamma)$.  Our main result is the following theorem:

\begin{theorem}
\label{UpperVertexLikeL}
Let $L=\left\{b_1,b_2,\ldots,b_{|V_n|}\right\}$ be a basis for $B_n$, such that for any $i$ and for any $a\in\left(V_n\setminus\text{span}\{b_1,\ldots,b_{i-1}\}\right)$, we have $k_{b_i}\geq k_a$.  Then $L$ is an upper vertex-like basis.
\end{theorem}

The arguments involved in the proof of this theorem rely heavily on the chain of subspaces
\[F_{|V_{n-1}|}\subseteq F_{|V_{n-1}|-1}\subseteq\ldots F_2\subseteq F_1=B_n\]
given by
\[F_i=span\{v\in V_n : k_v\geq i\}.\]
Notice that $F_{|V_{n-1}|}$ is the subspace of $B_n$ generated by the vertices that annihilate all of $B_{n-1}$ by left multiplication.  This is exactly the collection of vertices with out-degree one.  Considering the presentation of $\kappa_v$ given in Lemma~\ref{StructureOfKappa}, we can see that $F_{{|V_{n-1}|-1}}$ is the subspace spanned by the vertices with out-degree one or two, and that in general, $F_{|V_{n-1}|-k}$ is the subspace spanned by the vertices with out-degree less than or equal to $k+1$. We will use this notation for the duration of this section.  

\begin{prop}
\label{Fi}
For $i=1,\ldots,|V_{n-1}|,$ let $F_i$ be defined as above.  Then we have
\[F_i=span\{a\in B_n : k_a\geq i\}\]
\end{prop}
\begin{proof}
Clearly, $F_i\subseteq span\{a\in B_n : k_a\geq i\}$.  To obtain the opposite inclusion, suppose $a\in B_n$ satisfies $k_a\geq i$.  We know from Proposition~\ref{Kappa} that
\[a\in span\{v\in V_n : k_v\geq k_a\}\subseteq F_i,\]
and so $span\{a\in B_n: k_a\geq i\}\subseteq F_i$, and our proof is complete.
\end{proof}

\begin{defn} 
We say that a basis $\{x_1,\ldots,x_{\dim(V)}\}$ for a vector space $V$ is \textbf{compatible} with a chain of subspaces
\[V_0\subseteq V_1\subseteq V_2\subseteq\ldots\subseteq V\]
if for each $i$, $span\{x_1,\ldots,x_{\dim(V_i)}\}=V_i$.
\end{defn}

\begin{prop}
Let $L$ be defined as in Theorem~\ref{UpperVertexLikeL}.  Then $L$ is compatible with
\[F_{|V_{n-1}|}\subseteq F_{|V_{n-1}|-1}\subseteq\ldots F_2\subseteq F_1=B_n.\]
\end{prop}
\begin{proof}
Recall that $F_i$ is spanned by $\{v\in V_n: k_v\geq i\}$.  Thus if $j<\dim(F_i)$, then there exists $v\in F_i\setminus span\{b_1,\ldots,b_j\}$ satisfying $k_v\geq i$.  By the definition of $L$, we have $k_{b_{j+1}}\geq k_v\geq i$.  It follows that $b_{j+1}\in F_i$.  We conclude that
\[span\{b_1,\ldots,b_{\dim(F_i)}\}\subseteq span\{a\in B_n : k_a\geq i\}=F_i,\]
and so
\[span\{b_1,\ldots,b_{\dim(F_i)}\}=F_i\]
\end{proof}

This is sufficient to prove Proposition~\ref{OutDegree}, but it falls short of giving us a construction for an upper vertex-like basis. According to  Proposition~\ref{UpperVertexLikeBijection}, an upper vertex-like basis for $B_n$ consists of a linearly independent collection of algebra elements whose $\kappa$-subspaces ``match'' the $\kappa$-subspaces of the elements of $V_n$.  We still need to show that the $\kappa$-subspaces of $L$ appear as $\kappa$-subspaces associated to vertices.

\begin{prop}
Let $L$ be a basis for $B_n$ satisfying the hypothesis of Theorem ~\ref{UpperVertexLikeL}.  Then for each $b\in L$, there exists $w\in V_n$ such that $\kappa_b=\kappa_w$.
\end{prop}
\begin{proof}
By Proposition~\ref{Kappa}, any element $a\in B_n$ is in $span\{v\in V_n : \kappa_v\supseteq \kappa_a\}$.  Thus either there exists $w\in V_n$ with $\alpha_w\neq 0$ and $\kappa_a=\kappa_w$, or we have
\[a\in span\{v\in V_n : \kappa_v\supsetneq\kappa_a\}\subseteq span\{v\in V_n : k_v\geq k_a+1\}= F_{(k_a+1)}\]
By Proposition~\ref{Fi}, this gives us
\[a\in span\{b\in L : k_b\geq k_a+1\},\]
and so $a\notin L$.

It follows that for any $b=\sum_{v\in V_n}\beta_vv\in L$, there exists $w\in V_n$ such that $\beta_w\neq 0$ and $\kappa_b=\kappa_w$.
\end{proof}

All that remains is to show that the multiplicity with which each $\kappa$-subspace appears in $L$ matches the multiplicity with which it appears in $V_n$.

\begin{prop}
\label{SameSizeYay}
Let $L$ be a basis for $B_n$, satisfying the hypothesis of Theorem~\ref{UpperVertexLikeL}.  Then for each $w\in V_n$, we have
\[|\{b\in L : \kappa_b=\kappa_w\}|= |\{v\in V_n : \kappa_v=\kappa_w\}|\]
\end{prop}
\begin{proof}
Notice that 
\[\left\{b+F_{(k_w+1)} : b\in L, k_b=k_w\right\}\]
and
\[\left\{v+V_{(k_w+1)} : v\in V_n, k_v=k_w\right\}\]
are both bases for the space $F_{(k_w+1)}/F_{k_w}$.  We know that for each $b\in L$ with $\kappa_b=\kappa_w$, we have $b\in span\{v\in V_n : \kappa_v\supseteq \kappa_w\}$, and so
\[span\left\{b+F_{k_w+1} : b\in L, \kappa_b=\kappa_w\right\}\subseteq span\left\{v + F_{k_w+1} : v\in V_n, \kappa_v=\kappa_w\right\}\]
A simple dimension argument allows us to conclude that these spans are equal, and thus
\[|\{b\in L : \kappa_b=\kappa_w\}|= |\{v\in V_n : \kappa_v=\kappa_w\}|\]
\end{proof}

Now we are ready to prove that $L$ is an upper vertex-like basis.

\begin{proof}[Proof of Theorem~\ref{UpperVertexLikeL}]
By Corollary~\ref{SameSizeYay}, we know that for every $w\in V_{n}$, we have
\[|\{b\in L : \kappa_b=\kappa_w\}| = |\{v\in V_n : \kappa_v = \kappa_w\}|\]
Thus there exists a bijective map $\phi: V_n\rightarrow L$ such that $\kappa_{\phi(v)}=\kappa_v$.
\end{proof}

\section{Uniqueness Results}
\label{chap:UniquenessResults}

Upper vertex-like bases allow us to find the collection of $\kappa$-subspaces associated to the vertices, and to count the multiplicity of each of these $\kappa$-subspaces.  However, there are cases in which we can do even better---identifying the linear span of a particular vertex in $B(\Gamma)$.  Recall that for any $a\in B_n$,
\[a\in span\{v\in V_n : \kappa_v\supseteq \kappa_a\}\]
Thus if there exists a vertex $w\in V_n$ such that
\[\{v\in V_n :\kappa_v\supseteq \kappa_w\}=\{w\},\]
then for any element $b$ of an upper vertex-like basis $L$ for $B_n$ satisfying $\kappa_b=\kappa_w$, we have $b\in span\{w\}$.  In cases where this condition is common, this allows us to completely determine the structure of $\Gamma$ from the algebra $B(\Gamma)$.  

We will use both poset and layered-graph notation in the discussion that follows.  We will identify the poset $P$ with the layered graph associated to its Hasse diagram.  For ease of notation, we will write $P_{\geq j}$ for $\bigcup_{i\geq j}P_i$.

\subsection{Non-Nesting Posets}
\label{sec:NonNesting}

\begin{defn}
Let $P$ be a ranked poset with a unique minimal vertex $*$.  We will say that $P$ has the \textbf{non-nesting property} if for any two distinct elements $p$ and $q$ with out-degree greater than 1, we have $S(p)\not\subseteq S(q)$.
\end{defn}

\begin{theorem}
\label{NonNesting}
Let $P$ be a finite poset with the non-nesting property, such that $|S(p)|>1$ whenever $|p|>1$.  If $Q$ is a poset satisfying $Q\sim_B P$, then $P_{\geq2}\cong Q_{\geq2}$
\end{theorem}

\begin{proof}
Suppose there is a doubly graded algebra isomorphism from $B(P)$ to $B(Q)$.  We will equate the vertices in $B(P)$ with their images in $B(Q)$, allowing us to work inside the algebra $B(Q)$.

For each level $i$, we can find an upper vertex-like basis $L_i$ for $B_i$.  The construction of $L_i$ depends on the algebra, not on the original poset, so $L_i$ is upper vertex-like for both $P$ and for $Q$.  That is to say, there exist two bijections
\[\phi_i:L_i\rightarrow P_i\hspace{12pt}\text{and}\hspace{12pt}\psi_i:L_i\rightarrow Q_i\]
such that $\kappa_{\phi_i(\mathcal{A})}=\kappa_\mathcal{A}$ and $\kappa_{\psi_i(\mathcal{A})}=\kappa_\mathcal{A}$ for every $\mathcal{A}\subseteq L_i$.

Define $\xi_i:P_i\rightarrow Q_i$ by $\xi_i=\psi_i\circ\phi_i\inv$.  This is a bijection, and for any $\mathcal{A}\subseteq P_i$, we have
\[\kappa_{\xi_i(\mathcal{A})}=\kappa_{\phi_i\inv(\psi_i(\mathcal{A}))} =\kappa_{\psi_i(\mathcal{A})}=\kappa_{\mathcal{A}}.\]

For any $q\in Q_i$, $a\in B_i$, we know from Theorem~\ref{Kappa} that $\kappa_a=\kappa_q$ only if 
\[a\in\text{span}\{r\in Q_i: \kappa_q\subseteq\kappa_r\}=\text{span}\{r\in Q_i: S(r)\subseteq S(q)\}\]

For $i>1$, the non-nesting property tells us that the rightmost set is equal to $span\{q\}$, and so $a$ is a scalar multiple of $q$.  In particular, for each $p\in P_i$, $\xi_i(p)$ is a scalar multiple of $p$.  We will write $\xi_i(p)=\alpha_pp$.  

Now define a bijection
\[\xi:P_{\geq 2}\rightarrow Q_{\geq 2}\]
such that $\xi(p)=\xi_i(p)$ for every $p\in P_i$.  We claim that this is an isomorphism of posets.  To prove this, we must show that for any $r\in P_{\geq 2},$ we have $r\in S(p)$ if and only if $\xi(r)\in S(\xi(p))$.

Let $p\in P_i$ for $i>2$.  We know that
\[\kappa_p=\text{span}\left(\left\{\sum_{r\in S(p)}r\right\}\cup\{r:r\notin S(p)\}\right),\]
and that $|S(p)|>1$.  It follows that for $r\in P_{i-1}$ we have $r\in S(p)$ if and only if $r\notin \kappa_p$.  A similar argument shows that $\xi(r)\in S(\xi(p))$ if and only if $\xi(r)\notin \kappa_{\xi(p)}$.  This gives us
\[\big(r\in S(p)\big) \Leftrightarrow \big(r\notin\kappa_p\big) \Leftrightarrow \big(\alpha_rr\notin\kappa_p\big) \Leftrightarrow \big(\xi(r)\notin\kappa_{\xi(p)}\big) \Leftrightarrow \big(\xi(r)\in S(\xi(p))\big)\]
\end{proof}

\begin{cor}
Given any finite atomic lattice $P$ whose Hasse diagram is a uniform layered graph, the poset $P_{\geq 2}$ is determined up to isomorphism by $B(P)$.
\end{cor}

\begin{proof}
Since $P$ is a finite lattice, it has a unique minimal element $\hat{0}$.  Since $P$ is atomic, any element $p$ of rank two or greater satisfies $|S(p)|>1$.  All that remains is to show that $P$ satisfies the non-nesting property.

Suppose $p$ and $q$ are elements of $P$ with $S(p)\subseteq S(q)$.  Then either $p$ and $q$ are both atoms and $S(p)=S(q)=\hat{0}$, or we have $\bigvee S(p)=p$ and $\bigvee S(q)=q$.  This means that
\[p\vee q=\left(\bigvee S(p)\right)\vee\left(\bigvee S(q)\right)=\bigvee S(q)=q,\]
and so we have $q\geq p$.  Since $S(p)=S(q)$, we know that $|p|=|q|$, and so it follows that $q=p$.
\end{proof}

\subsection{The Boolean Algebra}

The Boolean algebra satisfies the non-nesting property.  This allows us to show that it is uniquely identified by its algebra $B(\Gamma)$.

\begin{prop}
\label{BooleanAlgebra}
Let $2^{[n]}$ be the Boolean lattice, and let $\Gamma$ be a layered graph with $\Gamma\sim_B 2^{[n]}$.  Then $\Gamma$ and $2^{[n]}$ are isomorphic as layered graphs.
\end{prop}
\begin{proof}
Let $2^{[n]}$ be the Boolean lattice, and let $\Gamma=(V_0\cup\ldots\cup V_n,E)$ be a uniform layered graph with unique minimal vertex such that there exists an isomorphism of doubly-graded algebras from $B\left(2^{[n]}\right)$ to $B(\Gamma)$.  As in the proof of Theorem~\ref{NonNesting}, we equate the vertices in $B\left(2^{[n]}\right)$ with their images in $B(\Gamma)$.  Again following the proof of Theorem~\ref{NonNesting}, we define $\xi_i$ for each $i>2$. Since $2^{[n]}$ is a finite atomic lattice, the map
\[\xi:\left(2^{[n]}\right)\\_{\geq2}\rightarrow\Gamma_{\geq2}\]
given by $\xi(p)=\xi_i(p)$ for $p\in \binom{[n]}{i}$ is an isomorphism of posets.  We would like to define an extension $\xi^\prime$ of $\xi$, such that $\xi^\prime$ is an isomorphism from $2^{[n]}$ to $\Gamma$.

We will use the following in our construction:

\textit{Claim 1:}  For every $w\in V_1$, we have $|\{v\in V_2 : v\ngtrdot w\}|=\binom{n-1}{2}$.

\textit{Claim 2:} For every $A\subseteq V_2$ with $|A|=\binom{n-1}{2}$, we have $\dim(\kappa_A)>1$ if and only if there exists $w\in V_1$ with $A=\{v\in V_2 : v\ngtrdot w\}$.

We will begin by proving Claim 1.  We know that $\xi$ satisfies $\kappa_{\xi(p)}=\kappa_p$ for any $p\in \binom{[n]}{2}$.  Since $|S(p)|=2$ for every $p\in\binom{[n]}{2}$, Proposition~\ref{OutDegree} gives us
\[\dim(\kappa_p)=n-1\]
for all $p\in\binom{[n]}{2}$, and thus
\[\dim(\kappa_v)=n-1\]
for all $v\in V_2$.  This tells us that $|S(v)|=2$ for all $v\in V_2$.

Furthermore, for any $p\neq q$ in $\binom{[n]}{2}$, we have $\kappa_p\neq\kappa_q$.  Since $\xi$ is a bijection, this tells us that $\kappa_v\neq \kappa_w$ for any $v\neq w$ in $V_2$.  Thus $S(v)\neq S(w)$ whenever $v\neq w$ in $V_2$.  We have
\[|V_1|=\dim(B_1)=n,\]
and
\[|V_2|=\dim(B_2)=\binom{n}{2},\]
so each possible pair of vertices in $V_1$ appears exactly once as $S(v)$ for some $v\in V_1$.  It follows that for any $w\in V_1$, $|\{v\in V_2 : v\ngtrdot w\}|$ is exactly the number of pairs of vertices in $V_1$ that exclude $w$.  There are exactly $\binom{n-1}{2}$ such pairs.  This proves Claim 1.

To prove Claim 2, we begin by associating to each subset $A\subseteq V_2$ a graph $G_A$ with vertex set $V_1$, and with edge set
\[\{S(v): v\in A\}.\]
$G_A$ is a graph with $|V_1|$ vertices, $|A|$ edges, and $k_A$ connected components.  If the connected components have sizes $i_1, i_2,\ldots, i_{k_A}$, then we have
\[|A|\leq \binom{i_1}{2}+\binom{i_2}{2}+\ldots+\binom{i_{k_A}}{2}\]
Assume that $|A|=\binom{n-1}{2}$. Then
\[\binom{n-1}{2}\leq \binom{i_1}{2}+\binom{i_2}{2}+\ldots+\binom{i_{k_A}}{2}.\]

Notice that for $i\geq j\geq 1$, we have
\begin{eqnarray*}
\binom{i}{2}+\binom{j}{2}&=&\frac{i(i-1)+j(j-1)}{2}\\
&\leq& \frac{i(i-1)+i(j-1)}{2}\\
&\leq& \frac{i(i-1)+(j-1)(j-2)+i(j-1)}{2}\\
&\leq& \binom{i}{2}+\binom{j-1}{2}+i(j-1)\\
&=&\binom{i+j-1}{2},
\end{eqnarray*}
with equality if and only if $j=1$.  By induction,
\[\binom{i_1}{2}+\binom{i_2}{2}+\ldots+\binom{i_{k_A}}{2}\leq\binom{n-k_A+1}{2},\]
with equality if and only if all but one of $i_1,\ldots,i_k$ is equal to 1.

 This gives us
\[|A|=\binom{n-1}{2}\leq\binom{n+1-k_A}{2},\]
which implies that $k_A\leq 2$.  Thus either $\dim(\kappa_A)=1$, or $\dim(\kappa_A)=2$, in which case $k_A=2$ and so we have
\begin{eqnarray*}
\binom{n-1}{2}&\leq&\binom{i}{2}+\binom{j}{2}\\
&\leq&\binom{i+j-1}{2}\\
&=&\binom{n-1}{2},
\end{eqnarray*}
where $i$ and $j$ are the sizes of the two connected components of $G_A$, with $i\geq j$.  The inequality in the second line is an equality if and only if $j=1.$  Thus when $|A|=\binom{n-1}{2}$ and $k_A>1$, $G_A$ consists of one isolated vertex $w$, together with the complete graph on $V_1\setminus\{w\}$.  So in this case, $A$ is exactly the collection of vertices $\{v\in V_2 : v\ngtrdot w\}$, which proves Claim 2.

Now we are prepared to prove the theorem.  For each $i\in[n]$, let $\mathscr{A}_i=\binom{[n]\setminus\{i\}}{2}$, and let $A_i=\xi(\mathscr{A}_i)$.  Since $\kappa_{A_i}=\kappa_{\mathscr{A}_i}$, we have $k_{A_i}=k_{\mathscr{A}_i}=2$.  Since $|A_i|=|\mathscr{A}_i|=\binom{n-1}{2}$, Claim 2 tells us that for each $i\in [n]$,  there exists a unique $w_i\in V_1$ such that $A_i~=~\{v\in~V_2~:~v\ngtrdot~w_i\}$.  We define our function $\xi^\prime:2^{[n]}\rightarrow\Gamma$ so that
\[\xi^\prime(p)=\left\{
\begin{array}{ll}
*&\text{if }p=\emptyset\\
w_i&\text{if }p=\{i\}\\
\xi(p)&\text{else}
\end{array}
\right.\]
We wish to show that this function is an isomorphism of posets. 

 First, notice that for $i\neq j$, we have $\kappa_{\mathscr{A}_i}\neq \kappa_{\mathscr{A}_j}$.  This implies that $A_i\neq A_j$, and so $w_i\neq w_j$.  Thus $\xi^\prime$ is a bijection.

To show that $\xi^{\prime}$ preserves order, we must show that for every $i\in [n]$ and $p\in\binom{[n]}{2}$, we have $w_i\lessdot\xi^\prime(p)$ if and only if $i\in p$.  If $i\in p$, then $\xi^\prime(p)\notin A_i$.  Since $A_i=\{v\in V_2 : v\ngtrdot w_i\}$, it follows that $w_i\lessdot\xi^\prime(p)$.  Conversely, if $i\notin p$, then $\xi^\prime(p)\in A_i$, and so $w_i\nlessdot \xi^\prime(p)$.  It follows that $\xi^\prime$ is an isomorphism of partially ordered sets, and so $\Gamma\cong 2^{[n]}$.
\end{proof}

\begin{cor}
Let $2^{[n]}$ be the Boolean lattice, and let $\Gamma$ be a layered graph with $\Gamma\sim_A 2^{[n]}$.  Then $\Gamma$ and $2^{[n]}$ are isomorphic as layered graphs.
\end{cor}

\begin{proof}
Since $\Gamma\sim_A2^{[n]}$, we know that both $gr A(\Gamma)$ and $gr A(2^{[n]})$ are quadratic, and that their quadratic duals $B(\Gamma)$ and $B(2^{[n]})$ are equal.  This gives us $\Gamma\sim_B 2^{[n]},$ and the result follows.
\end{proof}

\subsection{Subspaces of a Finite-Dimensional Vector Space over $\mathbb{F}_q$}
\begin{prop}
Let $X$ be an $n$-dimensional vector space over $\mathbb{F}_q$, the finite field of order $q$, and let $P_X$ be the set of subspaces of $X$, partially ordered by inclusion.  Let $\Gamma$ be a layered graph with $\Gamma\sim_B P_X$.  Then $\Gamma$ and $P_X$ are isomorphic as layered graphs.
\end{prop}

\begin{proof}
The case where $n\leq 2$ is trivial, so we will assume that $n\geq 3$.  We will use $P_i$ to denote the collection of $i$-dimensional subspaces of $X$.  Let $\Gamma=(V_0\cup V_1\cup\ldots\cup V_n,E)$ be a uniform layered graph with unique minimal vertex such that there exists a doubly graded algebra isomorphism from $B(P_X)$ to $B(\Gamma)$.  Once again we will equate the vertices in $B(P_X)$ with their images in $B(\Gamma)$, and we will define $\xi_i$ as in the proof of Theorem~\ref{NonNesting}.  Again our poset $P_X$ is a finite atomic lattice, so the map
\[\xi:(P_X)_{\geq2}\rightarrow\Gamma_{\geq2}\]
is an isomorphism of posets.  We wish to define an extension $\xi^\prime$ of $\xi$ that is an isomorphism from our whole poset $P_X$ to $\Gamma$.  Our proof will mirror that of Proposition~\ref{BooleanAlgebra}, and we will make use of the following two facts:

\textit{Claim 1:} For every $w\in V_1$, we have \[|\{v\in V_2: v\ngtrdot w\}|=\frac{(q^n-q^2)(q^{n-1}-1)}{(q-1)(q^2-1)}.\]

\textit{Claim 2:} For every $A\in V_2$ with \[|A|=\frac{(q^n-q^2)(q^{n-1}-1)}{(q-1)(q^2-1)},\] we have $k_A>1$ if and only if there exists $w\in V_1$ with $A=\{v\in V_2: v\ngtrdot w\}$.

We begin by proving Claim 1.  We know that $P_2$ consists of $\frac{(q^n-1)(q^{n-1}-1)}{(q-1)(q^2-1)}$ planes.  It follows that $V_2$ consists of $\frac{(q^n-1)(q^{n-1}-1)}{(q-1)(q^2-1)}$ vertices.  If we wish to show that
\[|\{v\in V_2: v\ngtrdot w\}|=\frac{(q^n-q^2)(q^{n-1}-1)}{(q-1)(q^2-1)},\]
it will suffice to show that
\[|\{v\in V_2 : v\gtrdot w\}| = \frac{(q^n-1)(q^{n-1}-1)}{(q-1)(q^2-1)}-\frac{(q^n-q^2)(q^{n-1}-1)}{(q-1)(q^2-1)}=\frac{(q^{n-1}-1)}{(q-1)}\]

Let $v,v^\prime\in V_2$ with $v\neq v^\prime$.  Since $\xi$ is a bijection, we have $\xi\inv(v)\neq\xi\inv(v^\prime)$.  Since two planes intersect in a unique line, we have
\[|S(\xi\inv(v))\cap S(\xi\inv(v^\prime))|=1.\]
By the argument given in the proof of Proposition~\ref{IntersectionSize}, this means that
\[\dim(B_n)-k_{\{\xi\inv(v)\}}-k_{\{\xi\inv(v^\prime)\}}+k_{\{\xi\inv(v),\xi\inv(v^\prime)\}}+1=1\]
Since $k_A=k_{\xi(A)}$ for any $A\subseteq P_2$, this gives us
\[\dim(B_n)-k_{\{v\}}-k_{\{v^\prime\}}+k_{\{v,v^\prime\}}+1=1,\]
and another application of Proposition~\ref{IntersectionSize} gives us
\[|S(v)\cap S(v^\prime)|\leq 1\]
for any $v\neq v^\prime$.  Each plane in $P_2$ contains $q+1$ lines.  It follows that each vertex $v\in V_2$ satisfies $|S(v)|=q+1$.

  Suppose we have $w\in V_1$ such that
\[|\{v\in V_2 : v\gtrdot w\}|>\frac{(q^{n-1}-1)}{(q-1)}.\]
Then there must exist at least
\[q\left(\frac{q^{n-1}-1}{q-1}+1\right)+1=\frac{q^n+q^2-q-1}{q-1}\]
distinct vertices in $V_1$.  However, we know that
\[|V_1|=\frac{q^n-1}{q-1}<\frac{q^n+q^2-q-1}{q-1}\]
for $n>2$, so our initial assumption that there exists $w$ satisfying
\[|\{v\in V_2 : v\gtrdot w\}|>\frac{(q^{n-1}-1)}{(q-1)}\]
must be false.  It follows that
\[|\{v\in V_2 : v\gtrdot w\}|\leq\frac{(q^{n-1}-1)}{(q-1)}\]
for all $w\in V_1$.  Since each plane contains $q+1$ lines, we have
\begin{eqnarray*}
\sum_{w\in V_1}|\{v\in V_2 : v\gtrdot w\}|&=&(q+1)|V_2|\\
&=&(q+1)\left(\frac{(q^n-1)(q^{n-1}-1)}{(q-1)(q^2-1)}\right)\\
&=&\left(\frac{(q^n-1)}{(q-1)}\right)\left(\frac{(q^{n-1}-1)}{(q-1)}\right)\\
&=&|V_1|\left(\frac{(q^{n-1}-1)}{(q-1)}\right)
\end{eqnarray*}
It follows that
\[|\{v\in V_2 : v\gtrdot w\}|=\frac{(q^{n-1}-1)}{(q-1)}\]
for each $w\in V_1$, proving Claim 1.

We know that any two planes in $P_2$ intersect in exactly one line.  Thus for any $\mathscr{A}\subseteq P_2$, we have
\[\kappa_\mathscr{A}=\text{span}\left(\left\{\sum_{l\in S(\mathscr{A})}l\right\}\cup \{w\in V_1 : w\notin S(\mathscr{A})\}\right),\]
and thus $\dim(\kappa_\mathscr{A})=|P_1|-|S(\mathscr{A})|+1$ for all $\mathscr{A}\subseteq P_2$.  In particular, when $|\mathscr{A}|=\frac{(q^n-q^2)(q^{n-1}-1)}{(q-1)(q^2-1)}$, we have $k_\mathscr{A}>1$ if and only if $\mathscr{A}$ is the collection of planes which do not contain a particular line $l$.  There are $\frac{(q^n-1)}{(q-1)}$ such subsets, one for each line.  This means that there are exactly $\frac{(q^n-1)}{(q-1)}$ subsets $A\subseteq V_2$ with $|A|=\frac{(q^n-q^2)(q^{n-1}-1)}{(q-1)(q^2-1)}$ and $k_A>1$.  Since there are $|V_1|=\frac{(q^n-1)}{(q-1)}$ subsets of the form $\{v\in V_1 : v\ngtrdot w\}$ for some $w\in V_1$, this must be the complete list of subsets $A\subseteq V_2$ with $|A|=\frac{(q^n-q^2)(q^{n-1}-1)}{(q-1)(q^2-1)}$ and $k_A>1$.  This proves Claim 2.

Now we can construct our extension $\xi^\prime.$  For each point $p\in P_1,$ define $\mathscr{A}_p$ to be the collection of lines in $P_2$ which do not contain $p$, and let $A_p=\xi(\mathscr{A}_p)$.  Since $\kappa_{A_p}=\kappa_{\mathscr{A}_p}$, we have $k_{A_p}=2$, and thus there exists a unique $w_p\in V_1$ such that $A_p=\{v\in V_2 : v\ngtrdot w_p\}$.  We define our function $\xi^\prime: P_X\rightarrow \Gamma$ as follows:
\[\xi^\prime(x)=\left\{
\begin{array}{ll}
*&\text{if }x=\emptyset\\
w_x&\text{if }x\in P_1\\
\xi(x)&\text{else}
\end{array}
\right.\]

Since $\kappa_{\mathscr{A}_p}\neq\kappa_{\mathscr{A}_q}$ whenever $p\neq q$ in $P_1$, we know that $\xi^\prime(p)\neq\xi^\prime(q)$, and thus $\xi^\prime$ is a bijection.  

All that remains is to show that for $p\in P_1$ and $l\in P_2$, we have $w_p\lessdot \xi^\prime(l)$ if and only if $p\in l$.  If $p\in l$, then $\xi^\prime(l)\notin A_p,$ and since $A_p=\{v\in V_2 : v\ngtrdot w_p\}$, it follows that $w_p\lessdot\xi^\prime(l)$.  Conversely, if $p\notin l$, then $\xi^\prime(l)\in A_p$, and so $w_p\nlessdot \xi^\prime(p)$.  It follows that $\xi^\prime$ is an isomorphism of partially ordered sets, and that $\Gamma\cong P_X$.

\begin{cor}
Let $X$ be an $n$-dimensional vector space over $\mathbb{F}_q$, the finite field of order $q$, and let $P_X$ be the set of subspaces of $X$, partially ordered by inclusion.  Let $\Gamma$ be a layered graph with $\Gamma\sim_A P_X$.  Then $\Gamma$ and $P_X$ are isomorphic as layered graphs.
\end{cor}

\end{proof}


\bibliographystyle{plain}
\bibliography{Invariants10}

\end{document}